\documentclass[oneside,english]{amsart}
\usepackage[T1]{fontenc}
\usepackage[latin9]{inputenc}
\usepackage{verbatim}
\usepackage{float}
\usepackage{amsthm}
\usepackage{amstext}
\usepackage{amssymb}
\usepackage{graphicx}
\usepackage{esint}

\makeatletter

\floatstyle{ruled}
\newfloat{algorithm}{tbp}{loa}
\providecommand{\algorithmname}{Algorithm}
\floatname{algorithm}{\protect\algorithmname}

\numberwithin{equation}{section}
\numberwithin{figure}{section}
\newenvironment{lyxcode}
{\par\begin{list}{}{
\setlength{\rightmargin}{\leftmargin}
\setlength{\listparindent}{0pt}
\raggedright
\setlength{\itemsep}{0pt}
\setlength{\parsep}{0pt}
\normalfont\ttfamily}%
 \item[]}
{\end{list}}
\theoremstyle{plain}
\newtheorem{thm}{\protect\theoremname}

\makeatother

\usepackage{babel}
\providecommand{\theoremname}{Theorem}

\title{An asymptotic parallel-in-time method for highly oscillatory PDEs} 

\author{Terry Haut \& Beth Wingate}

\begin{document}
\maketitle

\begin{abstract}
We present a new time-stepping algorithm for nonlinear PDEs that exhibit
scale separation in time of a highly oscillatory nature. The algorithm
combines the parareal method---a parallel-in-time scheme introduced
in \cite{LI-MA-TU}---with techniques from the Heterogeneous Multiscale
Method (HMM) (cf. \cite{E-ENG:2003}),
which make use of the slow asymptotic structure of the equations \cite{Majda1998}.

We present error bounds, based on the analysis in \cite{Gander} and \cite{Bal2005},
that demonstrate convergence of the method. A complexity analysis
also demonstrates that the parallel speedup increases arbitrarily
with greater scale separation. Finally, we demonstrate the accuracy
and efficiency of the method on the (one-dimensional) rotating shallow
water equations, which is a standard test problem for new algorithms
in geophysical fluid problems. Compared to exponential integrators
such as ETDRK4 and Strang splitting---which solve the stiff oscillatory
part exactly---we find that we can use coarse time steps
that are orders of magnitude larger (for a comparable accuracy), yielding
an estimated parallel speedup of approximately $100$ for physically
realistic parameter values. For the (one-dimensional) shallow water
equations, we also show that the estimated parallel speedup of this
\textquotedblleft{}asymptotic parareal method\textquotedblright{}
is more than a factor of 10 greater than the speedup obtained from
the standard parareal method.\end{abstract}




\section{Introduction}

We present a new algorithm to integrate nonlinear PDEs that exhibit
scale separation in time. We focus on time scale separation of a highly
oscillatory nature, where standard (explicit or implicit) time-stepping
methods often require time steps that are on the order of the fastest
oscillation to achieve accuracy. This type of equation arises in numerous
scientific applications, including the large-scale simulations of
the ocean and atmosphere that serve as the primary motivation of this
paper \cite{Charney1948,Charney1949}. 

In particular, we consider computing solutions to equations of the
form 
\begin{equation}
\frac{\partial\mathbf{u}}{\partial t}+\frac{1}{\epsilon}\mathcal{L}\mathbf{u}=\mathcal{N}\left(\mathbf{u}\right)+\mathcal{D}\mathbf{u},\,\,\,\mathbf{u}\left(0\right)=\mathbf{u}_{0},\label{eq:full eqn}
\end{equation}
where the linear operator $\mathcal{L}$ has pure imaginary eigenvalues,
the nonlinear term $\mathcal{N}\left(\mathbf{u}\right)$ is of polynomial
type, the operator $\mathcal{D}$ encodes some form of dissipation,
and $\epsilon$ is a small non-dimensional parameter. For notational
simplicity, we let $\mathbf{u}\left(t\right)$ denote the spatial
(vector-valued) function $\mathbf{u}\left(t,\cdot\right)=\left(u_{1}\left(t,\cdot\right),u_{2}\left(t,\cdot\right),\ldots\right)$.
The operator $\epsilon^{-1}\mathcal{L}$ results in temporal oscillations
on an order $\mathcal{O}\left(\epsilon\right)$ time scale, and generally
necessitates small time steps if standard numerical integrators are
used. 

Our approach for integrating (\ref{eq:full eqn}) uses a variant of
the parareal algorithm \cite{LI-MA-TU}, which is a parallel-in-time
method that relies on a cheap coarse solver for computing in serial
a solution with low accuracy, and a more expensive fine solver for
iteratively refining the solutions in parallel. The key novelty in
this paper is to replace the numerical coarse solution of the full
equations (\ref{eq:full eqn}) with a locally asymptotic approximation
of (\ref{eq:full eqn}). This enables us to effectively bypass the
Nyquist constraint imposed by the fastest oscillations, and acheive
much greater parallel speedup. Examples on the rotating shallow water
equations---a standard benchmark against which to test new algorithms
in geophysical fluid applications---demonstrate that this approach
holds promise for increasing the accuracy and speed of geophysical
fluid simulations. In fact, we find (see Section~\ref{sub:Numerical-experiments-on})
that this approach allows us to take step sizes $\Delta T\gg\epsilon$
that are significantly larger than if alternative schemes are used
for the coarse solver, including exponential integrators and split-step
methods (which solve the stiff linear terms exactly). In the context
of large-scale simulations of the ocean and atmosphere, the gains
achieved from spatial parallelization alone are beginning to saturate,
and the results in this paper are a preliminary effort toward achieving
greater efficiency.

We first describe the standard parareal method in more detail, and
some of the challenges for acheiving high parallel speedup for problems
of the form (\ref{eq:full eqn}). The basic approach of the parareal
method is to take large time steps $\Delta T$ in serial using a coarse
integrator of (\ref{eq:full eqn}), and to iteratively refine the
solutions in parallel using small time steps $\Delta t$ and a more
accurate integrator. This can result in significant speedup in real
(wall-clock) time if the parareal iterations converge rapidly, and
either the ratio $\Delta T/\Delta t$ of coarse and fine step sizes
is large, or the cost of the coarse solver is much cheaper than that
of the fine solver (see Section~\ref{sec:Error-and-complexity} for
more details). Early applications of the parareal method include simulations
of molecular dynamics \cite{Baffico2002}, the Navier Stokes equation
\cite{Fischer04aparareal}, and quantum control problems \cite{MadayTurinici2003};
additional references can be found in \cite{Staff:2003}. Although
the parareal method has been most widely used for parabolic-type PDEs,
it has also been analyzed and used for accurate simulation of first
and and second order hyperbolic systems (cf. \cite{FarhatChandesris:2003},
\cite{Gander:2008}, and \cite{Dai}). A recent variant of the parareal
method also allows for the accurate long-time evolution of Hamiltonian
systems \cite{Legoll2012}. Finally, general convergence results for
the parareal algorithm can be found in \cite{Gander,Bal2005} (\cite{Gander}
also numerically demonstrates convergence on the Lorenz equations,
which is of particular relevance to geophysical fluid problems).

Despite the many successes of the parareal method, a basic obstacle
remains for equations of the form (\ref{eq:full eqn}): namely, the
step size $\Delta T$ for a coarse integrator that is based on a standard
method generally must satisfy $\Delta T=\mathcal{O}\left(\epsilon\right)$
in order to achieve any accuracy at all (which is a prerequisite for
convergence of the parareal method). In practice, this can mean that
the coarse integrator in the parareal method must use very small time
steps for solving (\ref{eq:full eqn}), and the parallel speedup can
be minimal. There are, however, some types of highly oscillatory PDEs
where numerical integrators have been developed that can take much
larger time steps $\Delta T\gg\epsilon$ (cf. \cite{HairerLubichWanner:2010}).

In this paper, we use a numerically computed 'locally slow' solution
that is based on the underlying asymptotic structure of (\ref{eq:full eqn}),
and which can allow step sizes $\Delta T\gg\epsilon$ significantly
larger than the Nyquist constraint imposed by the $\mathcal{O}\left(\epsilon\right)$
temporal oscillations (and thus a potentially significant parallel
speedup). A basic observation behind efficiently constructing a slow
solution is that the solution $\mathbf{u}\left(t\right)$ to (\ref{eq:full eqn})
has the asymptotic approximation $\mathbf{u}\left(t\right)=\exp\left(-t/\epsilon\mathcal{L}\right)\overline{\mathbf{u}}\left(t\right)+\mathcal{O}\left(\epsilon\right)$
(cf. \cite{MAJDA:2003}, \cite{Majda1998}, \cite{Wingate2011}),
where the slowly varying function $\overline{\mathbf{u}}\left(t\right)$
satisfies a reduced equation of the form
\begin{equation}
\frac{\partial\overline{\mathbf{u}}}{\partial t}=\overline{\mathcal{N}}\left(\overline{\mathbf{u}}\right)+\overline{\mathcal{D}}\overline{\mathbf{u}},\,\,\,\overline{\mathbf{u}}\left(0\right)=\mathbf{u}_{0}.\label{eq:averaged equation}
\end{equation}
Here the nonlinear term $\overline{N}\left(\overline{\mathbf{u}}\right)$
is given by the time average 
\begin{equation}
\overline{\mathcal{N}}\left(\overline{\mathbf{u}}\left(t\right)\right)=\lim_{\text{T}\rightarrow\infty}\frac{1}{T}\int_{0}^{T}e^{s\mathcal{L}}\mathcal{N}\left(e^{-s\mathcal{L}}\overline{\mathbf{u}}\left(t\right)\right)ds.\label{eq:average of nonlinear term}
\end{equation}
We emphasize that the above time averaging is performed with $\overline{\mathbf{u}}\left(t\right)$
held fixed. Similarly, 
\begin{equation}
\overline{\mathcal{D}}\overline{\mathbf{u}}\left(t\right)=\lim_{\text{T}\rightarrow\infty}\frac{1}{T}\int_{0}^{T}\left(e^{s\mathcal{L}}\mathcal{D}e^{-s\mathcal{L}}\right)\overline{\mathbf{u}}\left(t\right)ds.\label{eq:Du, average}
\end{equation}
Note that $\overline{\mathbf{u}}\left(t\right)$, and its time derivatives,
are formally bounded independently of $\epsilon$, and thus significantly
larger time steps $\Delta T\gg\epsilon$ can be taken to evolve (\ref{eq:averaged equation}).
Section~\ref{sec:alg with no scale separation} also discusses how
a numerical integrator based on a finite version of the time averages
(\ref{eq:average of nonlinear term}) and (\ref{eq:Du, average})
can be interpreted as a smoothed type of integrating factor method,
and allows accuracy even when there is no scale sepatation in time
(i.e. $\epsilon=\mathcal{O}\left(1\right)$); this is useful when
the scale separation is localized in space, and where it is desirable
to have a time step that is constrained only by the slow dynamics.

Despite the many successes of the above averaging procedure in elucidating
important qualitative features (see e.g. \cite{Majda1998} and \cite{Wingate2011}
for geophysical fluid dynamics applications), in practice this approach
may not be accurate enough for moderately small values of $\epsilon$
(e.g. $\epsilon=10^{-2}$ is typical in geophysical fluid applications),
where the implicit constant hidden in the $\mathcal{O}\left(\epsilon\right)$
notation can be significant \cite{Smith2005}. Since the parameter
$\epsilon$ is typically fixed in idealized applications, the resulting
asymptotic approximation cannot be refined without some additional
approach. Another limitation is that the asymptotic approximation
(\ref{eq:averaged equation}) is generally only valid on an $\mathcal{O}\left(1\right)$
time interval, and this situation is usually not improved by adding
more terms in the asymptotic expansion. For many applications, it
is therefore necessary to refine this approach in order to approximate
(\ref{eq:full eqn}) with a given target accuracy and on longer time
intervals. 

Our approach for computing the asymptotic approximation (\ref{eq:averaged equation})---for
the purpose of constructing a slow solution---is based on evaluating
the time averages $\overline{\mathcal{N}}\left(\overline{\mathbf{u}}\left(t\right)\right)$
and $\overline{\mathcal{D}}\overline{\mathbf{u}}\left(t\right)$ numerically;
this approach has also been used in \cite{CastellaChartierFaou2009}
to solve Hamiltonian systems more efficiently (see also \cite{NadigaHechtMargolinSmolarkiewicz:1997}
and \cite{Jones1999} for related approaches in geophysical simulations).
More generally, our numerical scheme for (\ref{eq:averaged equation})
is an instance of the Heterogeneous Multiscale Method (HMM) (cf. \cite{E-ENG:2003, EngquistTsai:2005,ArielEngquistTsai:2009,ArielEngquistKimLeeTsai:2013}
for selective applications to highly oscillatory problems), which
is a very general framework for efficiently computing approximations
to problems that exhibit multiple spatial or temporal scales; a review
of HMM can be found in \cite{WeinanEngquistXiantao2007}. The basic
idea is that, by integrating in time against a carefully chosen smooth
kernel, the time average can be performed over a window of length
$T_{0}=T_{0}\left(\epsilon\right)\ll1/\epsilon$; therefore, the overall
cost of evolving (\ref{eq:averaged equation}) is asymptotically smaller
than the cost of computing (\ref{eq:full eqn}) directly, and can
lead to arbitrarily large efficiency gains. In problems arising in
geophysical fluid applications, the value $\epsilon$ may only be
moderately small (e.g. $\epsilon\approx10^{-2}$), and in such cases
numerically computing the average (\ref{eq:average of nonlinear term})
can be as costly as explicitly integrating the full equation (\ref{eq:full eqn}).
However, the numerical average can itself be performed in an embarassingly
parallel manner (see Section~\ref{sec:An-asymptotic-slow}), and
thus is not expected to impact the overall (wall-clock) speed of the
algorithm. Finally, we remark that for solutions that develop sharp
gradients, it may be necessary to use the modified version of the
parareal method explored in \cite{Dai}. However, this is beyond the
scope of this paper.

The idea of using a coarse solution based on a modified equation is
not new, and the possibility has been mentioned early on in the parareal
literature (cf. \cite{Staff:2003}). In \cite{Maday:2007} and \cite{Engblom:2009},
multi-scale versions of the parareal method are developed for, respectively,
deterministic and stochastic chemical kinetic simulations; in \cite{Engblom:2009},
the coarse solver is based on a deterministic (macroscopic) approximation.
A recent paper \cite{Legoll2012b} applies a version of the parareal
method to systems of ODEs that exhibit fast and slow components. The
multi-scale coarse solver in \cite{Legoll2012b} uses a projection
onto the slow, low-dimensional manifold, and examples are provided
for singularly perturbed ODEs with dissipative-type scale separation.
In contrast to \cite{Maday:2007}, \cite{Engblom:2009}, and \cite{Legoll2012b}
, here we investigate this procedure for a model nonlinear PDE whose
scale separation is of a highly oscillatory nature, and where methods
that work well for stiff dissipative problems (e.g. implicit or exponential
integrators) generally fail to impart significant speedup. Moreover,
the asymptotic approximation (\ref{eq:averaged equation}) to (\ref{eq:full eqn})
cannot be computed explicitly in most cases, and this necessitates
using additional techniques. The locally asymptotic solver developed
here works even when there is no scale separation, which is an important
feature when the time scale separation is a function of space and
time (as occurs in some geophysical fluid applications); in this case,
the time step in the coarse, asymptotic solution is only constrained
by the slow dynamics. 

In Section~\ref{sec:An-asymptotic-slow}, we present a version of
the Heterogeneous Multiscale Method that is appropriate for efficiently
computing the asymptotic approximation (\ref{eq:averaged equation}).
We then present in Sections~\ref{sec:Parareal-Method} a variant
of the parareal method that is based on replacing the coarse solver
with a locally asymptotic approximation. Section~\ref{sec:alg with no scale separation}
shows that, by averaging over a scale on which the slow dynamics is
occuring, this HMM-type coarse solution is able to achieve accuracy
even when there is no scale separation. We provide complexity bounds
for this algorithm, which demonstrate that the parallel speedup increases
arbitrarily as $\epsilon$ decreases. We also present error bounds
that are based on the analysis in \cite{Gander} and \cite{Bal2005},
and that demonstrate convergence of the method under reasonable assumptions.
Finally, Section~\ref{sub:Numerical-experiments-on} discusses some
numerical experiments on the (one-dimensional) rotating shallow water
equations, which serve as a standard first test for new numerical
algorithms in geophysical applications. Our experiments show that,
in contrast to standard versions of the parareal algorithm, the algorithm
can converge to high accuracy in few iterations even when large time
steps $\Delta T\gg\epsilon$ are taken for the coarse solution. In
fact, compared to the exponential time differencing method (ETDRK4),
the integrating factor method, and Strang splitting (cf. \cite{CoxMathews:2002},
\cite{KassamTrefethen:2005}, and \cite{Lawson:1967})---all of which
integrate the stiff linear term $\epsilon^{-1}\mathcal{L}$ exactly---our
algorithm yields an estimated parallel speedup of $\approx100$ for
the physically realistic value of $\epsilon=10^{-2}$. We also show
that, for $\epsilon=10^{-2}$, this parallel speedup is at least $10$
times greater than the speedup that can be achieved by using the standard
parareal method with ETDRK4, OIFS, or Strang splitting as the coarse
solver. Finally, we demonstrate that the asymptotic parareal method
yields high accuracy even when $\epsilon=1$ (i.e. in the absense
of scale separation), and with a parallel speedup that is comparable
to using the standard parareal method with ETDRK4, OIFS, or Strang
splitting as the coarse solver. Using a coarser spatial discretization
in the asymptotic solver may also result in even greater efficiency
gains.

\section{An asymptotic slow solution\label{sec:An-asymptotic-slow}}

We use the Heterogeneous Multiscale Method (HMM) to solve (\ref{eq:averaged equation}),
which relies on computing the averages (\ref{eq:average of nonlinear term})
and (\ref{eq:Du, average}) numerically (see also \cite{CastellaChartierFaou2009}).
The key idea is that, by averaging in time with respect to an appropriate
smooth kernel and over a carefully selected window length $T_{0}=T_{0}\left(\epsilon\right)$,
the cost is asymptotically smaller than $1/\epsilon$ (which is the
cost of solving the full equation (\ref{eq:full eqn}) on an $\mathcal{O}\left(1\right)$
time interval). Once such time averages can be computed, then large
step sizes $\Delta T\gg\epsilon$, coupled with a standard numerical
integrator, can be taken to evolve (\ref{eq:averaged equation}).
For simplicity, we restrict our discussion to computing the time average
(\ref{eq:average of nonlinear term}) (in fact, for the equations
we consider here, the operators $\mathcal{L}$ and $\mathcal{D}$
commute, and so the average $\overline{\mathcal{D}}$ in (\ref{eq:Du, average})
satisfies $\overline{\mathcal{D}}=\mathcal{D}$). 

The basic approach for computing the time average (\ref{eq:average of nonlinear term})
involves the following approximations:
\begin{eqnarray}
\overline{\mathcal{N}}\left(\overline{\mathbf{u}}\left(t\right)\right) & = & \lim_{\text{T}\rightarrow\infty}\frac{1}{T}\int_{0}^{T}e^{s\mathcal{L}}\mathcal{N}\left(e^{-s\mathcal{L}}\overline{\mathbf{u}}\left(t\right)\right)ds\nonumber \\
 & \approx & \frac{1}{T_{0}}\int_{0}^{T_{0}}\rho\left(\frac{s}{T_{0}}\right)e^{s\mathcal{L}}\mathcal{N}\left(e^{-s\mathcal{L}}\overline{\mathbf{u}}\left(t\right)\right)ds\label{eq:numerical average, discretized}\\
 & \approx & \frac{1}{\overline{M}}\sum_{m=0}^{\overline{M}-1}\rho\left(\frac{s_{m}}{T_{0}}\right)e^{s_{m}\mathcal{L}}\mathcal{N}\left(e^{-s_{m}\mathcal{L}}\overline{\mathbf{u}}\left(t\right)\right).\nonumber 
\end{eqnarray}
The smooth kernel $\rho\left(s\right)$, $0\leq s\leq1$, is chosen
so that the length $T_{0}=T_{0}\left(\epsilon\right)$ of the time
window over which the averaging is done is as small as possible, and
that the error introduced by using the trapezoidal rule is negligible
(see e.g. \cite{EngquistTsai:2005} and \cite{Weinan:2003} for an
error analysis). 

More formally, we define the finite time average 
\begin{equation}
\overline{\mathcal{N}}_{\rho,T_{0}}\left(\overline{\mathbf{u}}\left(t\right)\right)=\frac{1}{T_{0}}\int_{0}^{T_{0}}\rho\left(\frac{s}{T_{0}}\right)e^{s\mathcal{L}}\mathcal{N}\left(e^{-s\mathcal{L}}\overline{\mathbf{u}}\left(t\right)\right)ds.\label{eq:finite time average}
\end{equation}
Then we need to choose the kernel $\rho\left(s\right)$ and the parameters
$T_{0}=T_{0}\left(\epsilon\right)$ and $\overline{M}$ so that the
truncation error, 
\[
\left\Vert \overline{\mathcal{N}}\left(\overline{\mathbf{u}}\left(t\right)\right)-\overline{\mathcal{N}}_{\rho,T_{0}}\left(\overline{\mathbf{u}}\left(t\right)\right)\right\Vert ,
\]
and the discretization error,
\[
\left\Vert \overline{\mathcal{N}}_{\rho,T_{0}}\left(\overline{\mathbf{u}}\left(t\right)\right)-\frac{1}{\overline{M}}\sum_{m=0}^{\overline{M}-1}\rho\left(\frac{s_{m}}{T_{0}}\right)e^{\left(s_{m}/\epsilon\right)\mathcal{L}}\mathcal{N}\left(e^{-\left(s_{m}/\epsilon\right)\mathcal{L}}\overline{\mathbf{u}}\left(t\right)\right)\right\Vert ,
\]
are smaller than our desired approximation tolerance. This can be
accomplished by requiring that the kernel $\rho\left(s\right)$ satisfies
$\rho^{\left(m\right)}\left(0\right)=\rho^{\left(m\right)}\left(1\right)=0$,
$m=0,1,\ldots$ It will also be convenient for comparison with Section~\ref{sec:alg with no scale separation}
to change variables $s\rightarrow s/\epsilon$ in the integrand in
(\ref{eq:finite time average}) to obtain the equivalent form 
\[
\overline{\mathcal{N}}_{\rho,T_{0}}\left(\overline{\mathbf{u}}\left(t\right)\right)=\frac{1}{\epsilon T_{0}}\int_{0}^{\epsilon T_{0}}\rho\left(\frac{s}{\epsilon T_{0}}\right)e^{s/\epsilon\mathcal{L}}\mathcal{N}\left(e^{-s/\epsilon\mathcal{L}}\overline{\mathbf{u}}\left(t\right)\right)ds.
\]

To better understand the role of $\rho\left(s\right)$ and the parameters
$T_{0}$ and $\overline{M}$, we informally analyze the above averaging
procedure in more detail. First,
since we are assuming that the nonlinear operator $\mathcal{N}$ in
(\ref{eq:average of nonlinear term}) is of polynomial type and $\mathcal{L}$
has pure imaginary eigenvalues, we can (in principle) expand $\overline{\mathbf{u}}\left(t\right)$
in terms of eigenfunctions of $\mathcal{L}$ and express the nonlinear
term in the form 
\begin{eqnarray}
e^{s\mathcal{L}}\mathcal{N}\left(e^{-s\mathcal{L}}\overline{\mathbf{u}}\left(t\right)\right) & = & \sum_{\lambda_{n}}e^{i\lambda_{n}s}\mathcal{N}_{n}\left(\overline{\mathbf{u}}\left(t\right)\right)\nonumber \\
 & = & \sum_{\lambda_{n}=0}\mathcal{N}_{n}\left(\overline{\mathbf{u}}\left(t\right)\right)+\sum_{\lambda_{n}\neq0}e^{i\lambda_{n}s}\mathcal{N}_{n}\left(\overline{\mathbf{u}}\left(t\right)\right),\label{eq:decomp for nonlinear}
\end{eqnarray}
Here the pure imaginary numbers $i\lambda_{n}$ are linear combinations
of the eigenvalues of $\mathcal{L}$, and the set $\lambda_{n}=0$
corresponds to resonant interactions (see Section~\ref{sub:Rotating-Shallow-Water}
for a concrete example in the context of the rotating shallow water
equations). In fact, using the definition of the averaging $\overline{\mathcal{N}}\left(\overline{\mathbf{u}}\left(t\right)\right)$
operator and the above decomposition, we see that
\[
\overline{\mathcal{N}}\left(\overline{\mathbf{u}}\left(t\right)\right)=\sum_{\lambda_{n}=0}\mathcal{N}_{n}\left(\overline{\mathbf{u}}\left(t\right)\right).
\]
To compare this to the finite time average $\overline{\mathcal{N}}_{\rho,T_{0}}\left(\overline{\mathbf{u}}\left(t\right)\right)$,
use (\ref{eq:decomp for nonlinear}) to express this as
\begin{eqnarray*}
\overline{\mathcal{N}}_{\rho,T_{0}}\left(\overline{\mathbf{u}}\left(t\right)\right) & = & \frac{1}{T_{0}}\int_{0}^{T_{0}}\rho\left(\frac{s}{T_{0}}\right)e^{s\mathcal{L}}\mathcal{N}\left(e^{-s\mathcal{L}}\overline{\mathbf{u}}\left(t\right)\right)ds\\
 & = & \sum_{\lambda_{n}}\left(\int_{0}^{1}\rho\left(s\right)e^{i\lambda_{n}T_{0}s}ds\right)\mathcal{N}_{n}\left(\overline{\mathbf{u}}\left(t\right)\right).
\end{eqnarray*}
Comparing $\overline{\mathcal{N}}_{\rho,T_{0}}\left(\overline{\mathbf{u}}\left(t\right)\right)$
and $\overline{\mathcal{N}}\left(\overline{\mathbf{u}}\left(t\right)\right)$,
we therefore require that 
\begin{equation}
\int_{0}^{1}\rho\left(s\right)e^{i\lambda_{n}T_{0}s}ds\approx\begin{cases}
1, & \,\,\,\,\text{for}\,\,\,\,\lambda_{n}=0,\\
0, & \,\,\,\,\text{for}\,\,\,\,\lambda_{n}\neq0.
\end{cases}\label{eq:oscillatory terms to eliminate}
\end{equation}

In order to satisfy (\ref{eq:oscillatory terms to eliminate}) with
a time window length $T_{0}$ as small as possible, we choose a smooth
kernel $\rho\left(s\right)$ that satisfies $\rho^{\left(m\right)}\left(0\right)=\rho^{\left(m\right)}\left(1\right)=0$.
Repeated integration by parts shows that 
\[
\left|\int_{0}^{1}\rho\left(s\right)e^{i\lambda_{n}T_{0}s}ds\right|\leq C_{m}\left|\lambda_{n}T_{0}\right|^{-m},\,\,\,\, m=1,2,\ldots
\]
In particular, the above calculation indicates that choosing a time
window $T_{0}\left(\epsilon\right)=\epsilon^{-r}$, $0<r<1$, can
formally yield an error
\[
\left\Vert \overline{\mathcal{N}}\left(\overline{\mathbf{u}}\left(t\right)\right)-\overline{\mathcal{N}}_{\rho,T_{0}}\left(\overline{\mathbf{u}}\left(t\right)\right)\right\Vert =\mathcal{O}\left(\epsilon^{m}\right),\,\,\,\, m>1.
\]
Moreover, repeated integration by parts, coupled with the Euler\textendash{}Maclaurin
formula, also shows that the error induced from the trapezoidal rule
is similarly small,
\[
\left\Vert \overline{\mathcal{N}}_{\rho,T_{0}}\left(\overline{\mathbf{u}}\left(t\right)\right)-\frac{1}{M}\sum_{m=0}^{M-1}\rho\left(\frac{s_{m}}{T_{0}}\right)e^{\left(s_{m}/\epsilon\right)\mathcal{L}}\mathcal{N}\left(e^{-\left(s_{m}/\epsilon\right)\mathcal{L}}\overline{\mathbf{u}}\left(t\right)\right)\right\Vert =\mathcal{O}\left(\epsilon^{m}\right),\,\,\,\, m>1.
\]
One commonly used choice of kernel function is given by 
\[
\rho\left(s\right)=\begin{cases}
C\exp\left(-1/\left(s\left(1-s\right)\right)\right), & \,\,\,\,0<s<1,\\
0, & \,\,\,\,\left|s\right|\geq1,
\end{cases}
\]
where the constant $C$ is such so that $\left\Vert \rho\right\Vert _{1}=1$.

Notice that, compared to the asymptotic cost $\epsilon^{-1}$ of solving
the full equation (\ref{eq:full eqn}), arbitrarily large efficiency
gains are possible for the choice $T_{0}\left(\epsilon\right)=\epsilon^{-r}$,
$0<r<1$.

\section{The asymptotic Parareal Method\label{sec:Parareal-Method}}

We briefly review the parareal algorithm, in the context of replacing
the coarse solver with a numerically computed locally asymptotic solution
based on the asymptotic structure of the equations described by \cite{MAJDA:2003},
\cite{Majda1998}, \cite{Wingate2011}\cite{schochet1994,Wingate2011}. 

We suppose that we are interested in solving (\ref{eq:full eqn})
on the time interval $t\in\left[0,1\right]$. Let $\varphi_{t}\left(\mathbf{u}_{0}\right)$
denote the evolution operator associated with (\ref{eq:full eqn}),
so that $\mathbf{u}\left(t\right)=\varphi_{t}\left(\mathbf{u}_{0}\right)$
solves (\ref{eq:full eqn}). In a similar way, let $\psi_{t}\left(\mathbf{u}_{0}\right)$
denote the evolution operator associated with (\ref{eq:averaged equation}),
so that $\overline{\mathbf{u}}\left(t\right)=\psi_{t}\left(\mathbf{u}_{0}\right)$
solves (\ref{eq:averaged equation}). Finally, let $\overline{\varphi}_{t}\left(\mathbf{u}_{0}\right)$
denote the asymptotic approximation at time $t$ that results from
(\ref{eq:averaged equation}):
\[
\overline{\varphi}_{t}\left(\mathbf{u}_{0}\right)=e^{-\left(t/\epsilon\right)\mathcal{L}}\overline{\mathbf{u}}\left(t\right)=e^{-\left(t/\epsilon\right)\mathcal{L}}\psi_{t}\left(\mathbf{u}_{0}\right).
\]
Therefore, we have that $\varphi_{t}\left(\mathbf{u}_{0}\right)-\overline{\varphi}_{t}\left(\mathbf{u}_{0}\right)=\mathcal{O}\left(\epsilon\right)$.

To describe the parareal method, we first divide the time interval
$\left[0,T\right]$ into $N$ subintervals $\left[n\Delta T,(n+1)\Delta T\right]$,
$n=0,\ldots,N-1$. Starting with the identity 
\[
\mathbf{U}_{n}=\overline{\varphi}_{\Delta T}\left(\mathbf{U}_{n-1}\right)+\left(\varphi_{\Delta T}\left(\mathbf{U}_{n-1}\right)-\overline{\varphi}_{\Delta T}\left(\mathbf{U}_{n-1}\right)\right),\,\,\,\,\mathbf{U}_{n}=\mathbf{u}\left(n\Delta T\right),
\]
the parareal method computes approximations $\mathbf{U}_{n}^{k}\approx\mathbf{U}_{n}$
by the iterative procedure: 
\begin{equation}
\mathbf{U}_{n}^{k}=\overline{\varphi}_{\Delta T}\left(\mathbf{U}_{n-1}^{k}\right)+\left(\varphi_{\Delta T}\left(\mathbf{U}_{n-1}^{k-1}\right)-\overline{\varphi}_{\Delta T}\left(\mathbf{U}_{n-1}^{k-1}\right)\right),\,\,\, k=1,2,\ldots\label{eq:HMM parareal iteration}
\end{equation}
At iteration level $k=0$, the slow approximation $\mathbf{U}_{n}^{0}=\overline{\varphi}_{\Delta T}\left(\mathbf{U}_{n-1}^{0}\right)$
is used. Notice that, at iteration level $k$, the quantities $\mathbf{U}_{n-1}^{k-1}$
in the difference $\varphi_{\Delta T}\left(\mathbf{U}_{n-1}^{k-1}\right)-\overline{\varphi}_{\Delta T}\left(\mathbf{U}_{n-1}^{k-1}\right)$
are already computed; consequently, the difference $\varphi_{\Delta T}\left(\mathbf{U}_{n-1}^{k-1}\right)-\overline{\varphi}_{\Delta T}\left(\mathbf{U}_{n-1}^{k-1}\right)$
can be computed in parallel for each $n$. Since the computation of
$\overline{\varphi}_{\Delta T}\left(\mathbf{U}_{n-1}^{k}\right)$
is inexpensive, the overall algorithm is also inexpensive in a parallel
environment if the iterates converge rapidly. This parareal method
is illustrated in Figure \ref{fig:parareal illustration}. We note
that the parareal method can be interpreted as an inexact Newton-type
iteration (cf. \cite{Gander}).

\begin{figure}
\caption{This figure illustrates the asymptotic parareal algorithm. The verticl
axes represents a typical prognostic variable such as $h$, the thickness
of the layer of fluid in the shallow water system. The pink line depicts
the asymptotic solution at the large time steps $n\Delta T$. The
blue line depicts the parallel-in-time, fine scale corrections, using
small time steps $\Delta t\ll\Delta T$. Finally, the black line depicts
the updated solution at the large time steps $n\Delta T$.\label{fig:parareal illustration}}

\includegraphics[scale=0.6]{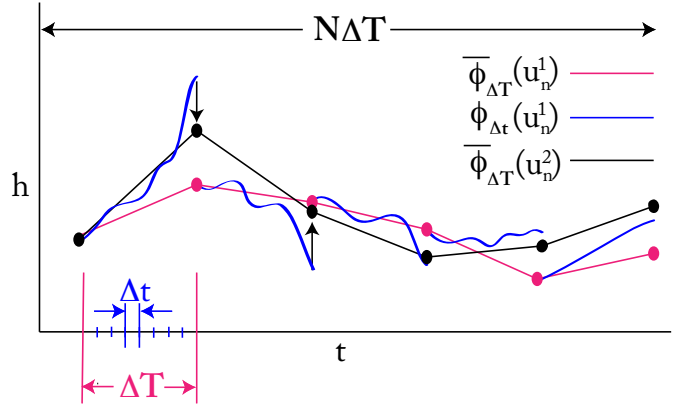}
\end{figure}

Full pseudocode for the asymptotic parareal method is presented below.
In the pseudocode, we use the mid-point rule in the HMM-type scheme
for the slow integrator, and Strang splitting for the fine integrator.
We assume that $\varphi_{\Delta T}\left(\mathbf{u}_{0}\right)$ is
computed using $M$ small time steps $\Delta t$ (so that $\Delta T=M\Delta t$),
and that $\overline{\varphi}_{\Delta T}\left(\mathbf{u}_{0}\right)$
is computed using one big time step $\Delta T$. 

\begin{algorithm}
\caption{Evaluate time average (in parallel)\label{alg:Evaluate-time-average}}

$\text{\ensuremath{\overline{\mathcal{N}}}}\left(\overline{\mathbf{u}}_{0}\right)$:
\begin{lyxcode}
parfor~$j=1,\ldots,\overline{M}-1$:

~~~~$s_{m}=T_{0}m/\overline{M}$

~~~~$\mathbf{u}_{m}\leftarrow\rho\left(s_{m}/T_{0}\right)e^{s_{m}\mathcal{L}}\mathcal{N}\left(e^{-s_{m}\mathcal{L}}\overline{\mathbf{u}}_{0}\right)$

end~parfor

$\mathbf{u}_{1}\leftarrow\text{Sum}\left(\mathbf{u}_{1},\ldots,\mathbf{u}_{\overline{M}}\right)$

\end{lyxcode}
\end{algorithm}

\begin{algorithm}
\caption{asymptotic slow solver}

$\text{Coarse\_Solver(}\mathbf{u}_{0},\Delta T)$:
\begin{lyxcode}
Take~a~$\Delta T/2$~timestep~for~the~linear~dissipative~term:
\[
\mathbf{v}\leftarrow e^{\left(\Delta T/2\right)\mathcal{D}}\mathbf{u}_{0}.
\]
Take~a~$\Delta T$~timestep~for~the~averaged~nonlinear~term:~
\begin{eqnarray*}
\mathbf{v} & \leftarrow & \overline{\mathcal{N}}\left(\mathbf{v}\right),\\
\mathbf{v} & \leftarrow & \overline{\mathcal{N}}\left(\mathbf{u}_{0}+\frac{\Delta T}{2}\mathbf{v}\right).
\end{eqnarray*}

Take~a~$\Delta T/2$~timestep~for~the~linear~dissipative~term:
\[
\mathbf{v}\leftarrow e^{\left(\Delta T/2\right)\mathcal{D}}\mathbf{v}.
\]

Transform~back~to~the~fast~time~coordinate:
\[
\mathbf{u}_{1}\leftarrow e^{\left(\Delta T/\epsilon\right)\mathcal{L}}\mathbf{v}.
\]

~Return~$\mathbf{u}_{1}$\end{lyxcode}
\end{algorithm}

\begin{algorithm}
\caption{Fine solver\label{alg:Fine-solver-Strang}}

$\text{Fine\_Solver(}\mathbf{u}_{0},\Delta t,\Delta T)$:
\begin{lyxcode}
$M=\Delta T/\Delta t$

for~$m=1,\ldots,M$:

~~~~take~$\Delta t/2$~timestep~for~the~linear~term:
\[
\mathbf{v}\leftarrow e^{\left(\Delta t/2\right)\left(\epsilon^{-1}\mathcal{L}+\mathcal{D}\right)}\mathbf{u}_{m}.
\]
~~~~take~a~$\Delta t$~timestep~for~the~nonlinear:~
\begin{eqnarray*}
\mathbf{v} & \leftarrow & \mathcal{N}\left(\mathbf{v}\right),\\
\mathbf{v} & \leftarrow & \mathcal{N}\left(\mathbf{u}_{m}+\frac{\Delta t}{2}\mathbf{v}\right).
\end{eqnarray*}

~~~~take~$\Delta t/2$~timestep~for~the~linear~term:
\[
\mathbf{u}_{m+1}\leftarrow e^{\left(\Delta t/2\right)\left(\epsilon^{-1}\mathcal{L}+\mathcal{D}\right)}\mathbf{v}.
\]

end~for

Return~$\mathbf{u}_{M}$.\end{lyxcode}
\end{algorithm}

\begin{algorithm}
\caption{Parallel-in-time integrator}

\begin{lyxcode}

Compute~the~initial~guess~using~the~slow~solver:

\textrm{$\mathbf{U}_{0}^{\text{old}}\leftarrow\mathbf{u}_{0}$}

for~$n=1,\ldots,N-1$

~~~~$\mathbf{U}_{n}^{\text{old}}\leftarrow\text{Coarse\_Solver(}\mathbf{U}_{n-1}^{\text{old}},\Delta T)$

endfor

Now~refine~the~solution~until~convergence:

\textrm{$\mathbf{U}_{0}^{\text{new}}\leftarrow\mathbf{u}_{0}$}

while~($\max_{n}\left\Vert \mathbf{U}_{n}^{\text{new}}-\mathbf{U}_{n}^{\text{old}}\right\Vert /\left\Vert \mathbf{U}_{n}^{\text{new}}\right\Vert >\text{tol}$)

~~~~parfor~$n=1,\ldots,N-1$:

~~~~~~~~$\mathbf{U}_{n}^{\text{old}}\leftarrow\mathbf{U}_{n}^{\text{new}}$

~~~~~~~~$\mathbf{V}_{n}\leftarrow\text{Fine\_Solver(}\mathbf{U}_{n}^{\text{old}},\Delta t,\Delta T)$

\textrm{~~~~~~~~~~~~~~~~~~~$\mathbf{V}_{n}\leftarrow\mathbf{V}_{n}-\text{Coarse\_Solver(}\mathbf{U}_{n}^{\text{old}},\Delta T)$}

~~~~end~parfor

~~~~for~$n=1,\ldots,N-1$

~~~~~~~~$\mathbf{U}_{n}^{\text{new}}\leftarrow\text{Coarse\_Solver(}\mathbf{U}_{n-1}^{\text{new}},\Delta T)+\mathbf{V}_{n-1}$

~~~~endfor

end~while

return~$\mathbf{U}_{1}^{\text{new}},\ldots,\mathbf{U}_{N}^{\text{new}}$\end{lyxcode}
\end{algorithm}

\section{The parallel-in-time algorithm without scale separation\label{sec:alg with no scale separation}}

In geophysical fluid problems, it is often the case that the time
scale separation can change in space and time, and it is important
that this algorithm works even when there is no scale separation.
We give a heauristic derivation of the time average (\ref{eq:finite time average})
from a different point of view, which indicates that the coarse solution
yields accuracy even when $\epsilon=\mathcal{O}\left(1\right)$, as
long as the time average in (\ref{eq:finite time average}) is performed
over a time scale on which the dynamics of the slow nonlinear terms
are occurring. Figure~\ref{fig: Delta_T vs. epsilon dependence}
schematically depicts how the large time step $\Delta T$ varies as
a function of the scale separation parameter $\epsilon$, for the
asymptotic parareal method, the standard parareal method, and a typical
time-stepping method that is used in serial.

\begin{figure}
\caption{Schematic of $\Delta T$ as a function of $\epsilon$ for (a) the
asymptotic parareal method (solid blue line), (b) the standard parareal
method with a linearly exact coarse solver (dashed red line), and
(c) a typical time-stepping method used in serial\label{fig: Delta_T vs. epsilon dependence}}

\includegraphics[scale=0.6]{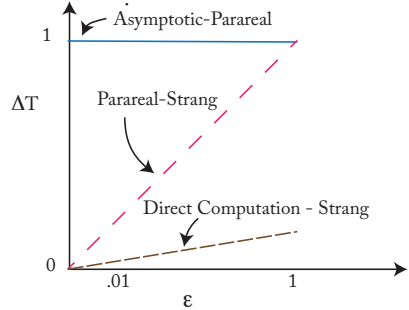}

\end{figure}

As in the integrating factor method, we first factor out the fast
oscillatory part,
\[
\mathbf{u}\left(t\right)=e^{-t/\epsilon\mathcal{L}}\mathbf{v}\left(t\right),
\]
so that $\mathbf{v}\left(t\right)$ satisfies 
\begin{equation}
\frac{\partial\mathbf{v}}{\partial t}=e^{t/\epsilon\mathcal{L}}\mathcal{N}\left(e^{-t/\epsilon\mathcal{L}}\mathbf{v}\left(t\right)\right).\label{eq:IF form, notes}
\end{equation}
Since 
\[
\frac{\partial\mathbf{v}}{\partial t}=\mathcal{O}\left(1\right),
\]
$\mathbf{v}\left(t\right)$ varies more slowly than $\mathbf{u}\left(t\right)$
and thus time steps $\Delta T\gg\epsilon$ can potentially be used
to solve for $\mathbf{v}\left(t\right)$. However, simply using a
standard time-stepping scheme for $\mathbf{v}\left(t\right)$ will
still require small step sizes. In fact, differentiating the equation
(\ref{eq:IF form, notes}) shows that $\mathbf{v}\left(t\right)$
has small but rapid fluctuations, 
\[
\frac{\partial^{2}\mathbf{v}}{\partial t^{2}}=\mathcal{O}\left(\frac{1}{\epsilon}\right),
\]
and standard time-stepping schemes will not be accurate unless $\Delta T$
is small.

\begin{figure}

\caption{Schematic depiction of the moving time average\label{fig:Schematic-depiction-of-the-moving-time-average}}

\includegraphics[scale=0.6]{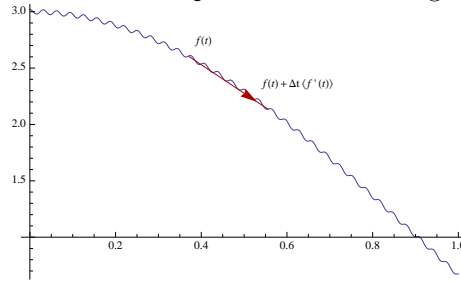}

\end{figure}

The idea (see also \cite{ArielEngquistBjörnKreissTsai:2009} and \cite{Reich1995})
is to take a time step using a smoothed out version of the derivative;
this is accomplished by using a moving time average, so that the small
$\mathcal{O}\left(\epsilon\right)$ fluctuations in the derivative
are removed; see Figure~\ref{fig:Schematic-depiction-of-the-moving-time-average}
for a schematic depiction. In particular, we average the derivative
out over a time scale $\epsilon T_{0}$ on which the slow dynamics
occur (that is, we average over $T_{0}$ fast oscillations). Then
using e.g. forward Euler we get the approximation
\[
\mathbf{v}\left(\Delta T\right)\approx\mathbf{v}\left(0\right)+\Delta T\left\langle \frac{\partial\mathbf{v}}{\partial t}\left(0\right)\right\rangle _{T_{0}},
\]
where
\begin{eqnarray*}
\left\langle \frac{\partial\mathbf{v}}{\partial t}\left(0\right)\right\rangle _{T_{0}} & = & \frac{1}{\epsilon T_{0}}\int_{0}^{\epsilon T_{0}}\rho\left(\frac{s}{\epsilon T_{0}}\right)\frac{\partial\mathbf{v}}{\partial s}\left(s\right)ds\\
 & \approx & \frac{1}{\epsilon T_{0}}\int_{0}^{\epsilon T_{0}}\rho\left(\frac{s}{\epsilon T_{0}}\right)e^{s/\epsilon\mathcal{L}}\mathcal{N}\left(e^{-s/\epsilon\mathcal{L}}\mathbf{v}\left(0\right)\right)ds.
\end{eqnarray*}
This approximation allows error control via two different mechanisms.
First, when $\epsilon\ll1$ the above approximation is also an asymptotic
approximation, and the time average serves to eliminate secular terms
that arise from nonlinear resonances as long as $\Delta T\gg\epsilon$.
However, when $\epsilon=\mathcal{O}\left(1\right)$, then we can take
$\epsilon T_{0}=\Delta T$ , so that the time average is performed
over a scale $\Delta T$ on which the dynamics is slow. In this case,
we see that the above approximation is essentially the forward Euler
method. In fact, in this case the derivatives of $\mathbf{v}\left(s\right)$
are slow on an $\mathcal{O}\left(\Delta T\right)$ time scale, and
we can Taylor expand to get
\begin{eqnarray*}
\frac{1}{\Delta T}\int_{0}^{\Delta T}\rho\left(\frac{s}{\Delta T}\right)\frac{\partial\mathbf{v}}{\partial s}\left(s\right)ds & \approx & \frac{1}{\Delta T}\int_{0}^{\Delta T}\rho\left(\frac{s}{\Delta T}\right)\left(\frac{\partial\mathbf{v}}{\partial s}\left(0\right)+s\frac{\partial^{2}\mathbf{v}}{\partial s^{2}}\left(0\right)+\dots\right)ds\\
 & = & \frac{\partial\mathbf{v}}{\partial s}\left(0\right)+\mathcal{O}\left(\Delta T\right).
\end{eqnarray*}
For systems of ODEs, rigorous error bounds for this ``partial time
averaging'' are derived in Section~$3.2$ of \cite{SandersVerhulst:1985}.

\section{Error and complexity bounds\label{sec:Error-and-complexity}}

We first discuss the complexity of our algorithm. Although this analysis
is standard for the parareal method, the complexity bounds demonstrate
(in theory) arbitrarily large parallel speedup as the parameter $\epsilon$
gets smaller. We assume that the time interval $\left[0,1\right]$
is divided into $N$ sub-intervals $\left[T_{n-1},T_{n}\right]$ of
length $\Delta T=T_{n}-T_{n-1}=1/N$. We also assume that, within
each subinterval, $M$ time steps of size $\Delta t$ are needed for
the fine integrator, so that $M=\Delta T/\Delta t$. We let $\tau_{\text{c}}$
denote the (wall-clock) time of computing the coarse solution, $\overline{\varphi}_{\Delta T}\left(\mathbf{u}_{0}\right)$,
over a large time step $\Delta T$. Similarly, we let $\tau_{\text{f}}$
denote the (wall-clock) time of computing the fine solution, $\varphi_{\Delta t}\left(\mathbf{u}_{0}\right)$,
over a small time step $\Delta t$.

To obtain an initial guess for the parareal method, we first compute
the slow approximations $\mathbf{U}_{n}^{0}=\overline{\varphi}_{\Delta T}\left(\mathbf{U}_{n-1}^{0}\right)$,
$n=1,\ldots,N$, which takes a wall-clock time of $\tau_{c}N$. Next,
suppose that we are at a given iteration level $k$, and we need to
compute the next iterations $\mathbf{U}_{n}^{k+1}$ from $\mathbf{U}_{n}^{k}$.
To do so, we first need to compute, in parallel, the difference $\mathbf{V}_{n}^{k}=\varphi_{\Delta T}\left(\mathbf{U}_{n}^{k}\right)-\overline{\varphi}_{\Delta T}\left(\mathbf{U}_{n}^{k}\right)$
between the coarse and fine solutions. This takes a wall-clock time
of $\tau_{\text{c}}+\tau_{\text{f}}M$. We then need to compute, in
serial, the updated approximations $\mathbf{U}_{n}^{k+1}=\overline{\varphi}_{\Delta T_{n-1}}\left(\mathbf{U}_{n-1}^{k+1}\right)+\mathbf{V}_{n-1}^{k},\,\,\,\, n=1,\ldots,N-1$.
This takes a wall-clock time of $\tau_{\text{c}}N$, for an overall
cost of $\tau_{\text{f}}M+\tau_{\text{c}}N+\tau_{\text{c}}$ per iteration.
Thus, after $\nu$ iterations, the overall cost of the parareal method
is $\nu\left(\tau_{\text{f}}M+\tau_{\text{c}}N+\tau_{\text{c}}\right)+\tau_{\text{c}}N$.
In contrast, directly solving (\ref{eq:full eqn}) in serial requires
time $\tau_{\text{f}}NM$. Thus, the estimated parallel speedup is
\begin{equation}
\frac{\tau_{\text{f}}NM}{\nu\left(\tau_{\text{f}}M+\tau_{\text{c}}N+\tau_{\text{c}}\right)+\tau_{\text{c}}N}\leq\min\left\{ \frac{\tau_{\text{f}}}{\tau_{\text{c}}}\frac{M}{\left(\nu+1\right)},\frac{N}{\nu}\right\} .\label{eq:speedup, upper bound}
\end{equation}
Notice that an upper bound on the speedup is proportional to $M=\Delta T/\Delta t$,
the ratio of slow to fine time step sizes.

In order for the asymptotic parareal method to converge, the fine
solver needs time steps $\Delta t$ that are some fraction $a$ of
$\epsilon$, $\Delta t=a\epsilon$, where $0<a<1$. Therefore, $\Delta T=M\Delta t=aM\epsilon$,
and $NM=\left(N\Delta T/\Delta t\right)=1/\left(a\epsilon\right)$.
Now, since
\[
\frac{\tau_{\text{f}}NM}{\nu\left(\tau_{\text{f}}M+\tau_{\text{c}}N\right)+\tau_{\text{c}}N}\approx\frac{1}{\tau_{\text{c}}\nu}\frac{\left(\tau_{\text{c}}N\right)\left(\tau_{\text{f}}M\right)}{\tau_{\text{f}}M+\tau_{\text{c}}N},
\]
this suggests taking $\tau_{\text{c}}N=\tau_{\text{f}}M$ (this choice
also balances the two terms in the upper bound (\ref{eq:speedup, upper bound})).
Therefore,
\[
N=\sqrt{\frac{\tau_{\text{f}}}{\tau_{\text{c}}}}\sqrt{\frac{1}{a\epsilon}},\,\,\,\, M=\sqrt{\frac{\tau_{\text{c}}}{\tau_{f}}}\sqrt{\frac{1}{a\epsilon}},
\]
and the estimated parallel speedup is given by
\begin{eqnarray*}
\frac{\tau_{\text{f}}NM}{\nu\left(\tau_{\text{f}}M+\tau_{\text{c}}N\right)+\tau_{\text{c}}N} & = & \frac{N\left(\frac{\tau_{\text{c}}}{\tau_{\text{f}}}N\right)}{\nu\left(\frac{\tau_{\text{c}}}{\tau_{\text{f}}}N+\frac{\tau_{\text{c}}}{\tau_{\text{f}}}N\right)+\frac{\tau_{\text{c}}}{\tau_{\text{f}}}N}\\
 & = & \frac{1}{2\nu+1}\sqrt{\frac{\tau_{\text{f}}}{\tau_{\text{c}}}}\sqrt{\frac{1}{a\epsilon}},
\end{eqnarray*}
which results in arbitrarily speedup relative to standard numerical
integrators.  %

Recall that computing a time step for the coarse solver, $\overline{\varphi}_{\Delta T}\left(\mathbf{u}_{0}\right)$,
requires evaluating the time average (\ref{eq:numerical average, discretized}). Since this can
be performed in an embarassingly parallel manner, the (wall-clock) time $\tau_{\text{c}}$ 
satisfies $\tau_{\text{c}} \sim \log_2 \left(  \overline{M} \right) $, where $\overline{M}$ is the number  of terms taken in the
time average (\ref{eq:numerical average, discretized}) (see Algorithm~\ref{alg:Evaluate-time-average}).
Also, as discussed in Section~\ref{sec:An-asymptotic-slow}, $\overline{M} \sim \epsilon^{-r} $, where $0<r<1$, and so the
estimated parallel speedup scales
like $ \left( \epsilon  \log_2 \left(  \epsilon^{-1} \right) \right)^{-1/2}  $. In particular, this 
results in arbitrarily large efficiency gains relative to standard numerical
integrators as $\epsilon \rightarrow 0$.  Note that, since $r<1$, arbitrary speedup is also possible without evaluating  
the time average (\ref{eq:numerical average, discretized})
in parallel; however, we have found that this may be necessary for obtaining satisfactory speedup when
$\epsilon$ is only moderately small. We remark that evaluating the average in parallel requires more processors,
and therefore decrease the parallel efficiency relative to standard parareal (in the experiments in Section~\ref{sec:Numerical-examples},
evaluating the average in parallel requires about a factor of $1.3$ or less extra processors).

We now use the theory developed in \cite{Bal2005} (see also \cite{Gander})
for our error bounds. As in \cite{Bal2005}, we consider a scale of
Banach spaces $B_{0}\supseteq B_{1}\supseteq B_{2}\supseteq\cdots$,
where $B_{j}$ typically quantifies the degree of regularity (i.e.,
functions in $B_{j+1}$ are more smooth than functions in $B_{j}$).
This consideration is useful when obtaining error bounds for infinite-dimensional
systems. In fact, as discussed in \cite{Bal2005}, sufficient regularity
constraints can be needed on the initial condition in order to ensure
convergence of the parareal method (especially in the absence of numerical
or analytical dissipation, and on long time intervals). Also, rigorous
bounds between the solution to (\ref{eq:full eqn}), and its asymptotic
approximation obtained from (\ref{eq:averaged equation}), can require
controlling ``small denominators'' that can arise from near resonances;
this can often be achieved if the intial condition is sufficiently
smooth. For example, for the so-called primitive equations---which
are fundamental in geophysical fluid dynamics---rigorous error bounds
for the asymptotic approximation (\ref{eq:averaged equation}) have
been derived in \cite{PetcuTemamWirosoetisno:2005}. These bounds
demonstrate that the error in the asymptotic approximation is small,
as long as the initial condition is sufficiently smooth. 

In order to account for the error that arises from solving the slow
evolution equation (\ref{eq:averaged equation}) numerically with
a time-stepping algorithm, we first need to introduce some more notation.
In particular, let $\widetilde{\varphi}_{\Delta T}\left(\mathbf{u}_{0}\right)$
denote the evolution operator associated with solving the slow equation
(\ref{eq:averaged equation}) using an order $p$ time-stepping method,
so that $\widetilde{\varphi}_{\Delta T}\left(\mathbf{u}_{0}\right)$
is a numerical approximation to $\overline{\varphi}_{\Delta T}\left(\mathbf{u}_{0}\right)$.
We derive error bounds for the iteration
\begin{equation}
\mathbf{U}_{n}^{k}=\widetilde{\varphi}_{\Delta T}\left(\mathbf{U}_{n-1}^{k}\right)+\left(\varphi_{\Delta T}\left(\mathbf{U}_{n-1}^{k-1}\right)-\widetilde{\varphi}_{\Delta T}\left(\mathbf{U}_{n-1}^{k-1}\right)\right).\label{eq:parareal iteration for proof}
\end{equation}
Notice that, in contrast to the iteration (\ref{eq:HMM parareal iteration})
presented in Section~\ref{sec:Parareal-Method}, we explicitly include
the numerical approximation $\widetilde{\varphi}_{\Delta T}\left(\mathbf{U}_{n-1}^{k-1}\right)$
to $\overline{\varphi}_{\Delta T}\left(\mathbf{U}_{n-1}^{k-1}\right)$
in the analysis. It is also possible to include the error arising
from computing the fine solution $\varphi_{\Delta T}\left(\mathbf{U}_{n-1}^{k-1}\right)$
with a time-stepping scheme, but this requires no additional techniques
and is omitted for simplicity.

Define the operators $\mathcal{E}_{\varphi,\overline{\varphi}}\left(\cdot\right)$
and $\mathcal{E}_{\overline{\varphi},\widetilde{\varphi}}\left(\cdot\right)$,
\begin{equation}
\mathcal{E}_{\varphi,\overline{\varphi}}\left(\mathbf{u}_{0}\right)=\varphi_{\Delta T}\left(\mathbf{u}_{0}\right)-\overline{\varphi}_{\Delta T}\left(\mathbf{u}_{0}\right),\,\,\,\,\,\mathcal{E}_{\overline{\varphi},\widetilde{\varphi}}\left(\mathbf{u}_{1}\right)=\overline{\varphi}_{\Delta T}\left(\mathbf{u}_{0}\right)-\widetilde{\varphi}_{\Delta T}\left(\mathbf{u}_{0}\right).\label{eq:def of difference operators, proof}
\end{equation}
Then following \cite{Bal2005}, we make the following assumptions
on $\varphi_{\Delta T}\left(\cdot\right)$, $\overline{\varphi}_{\Delta T}\left(\cdot\right)$,
$\widetilde{\varphi}_{\Delta T}\left(\cdot\right)$, $\mathcal{E}_{\varphi,\overline{\varphi}}\left(\cdot\right)$,
and $\mathcal{E}_{\overline{\varphi},\widetilde{\varphi}}\left(\cdot\right)$:
\begin{enumerate}
\item The operators $\varphi_{t}\left(\cdot\right)$ and $\overline{\varphi}_{t}\left(\cdot\right)$
are uniformly bounded for $0\leq t\leq1$,
\begin{equation}
\left\Vert \varphi_{t}\left(\mathbf{u}_{0}\right)\right\Vert _{B_{j}}\leq C_{j}\left\Vert \mathbf{u}_{0}\right\Vert _{B_{j+1}},\,\,\,\,\left\Vert \overline{\varphi}_{t}\left(\mathbf{u}_{0}\right)\right\Vert _{B_{j}}\leq C_{j}\left\Vert \mathbf{u}_{0}\right\Vert _{B_{j+1}}.\label{eq:stability, proof}
\end{equation}

\item The asymptotic approximation is accurate in the sense that 
\begin{equation}
\left\Vert \varphi_{t}\left(\mathbf{u}_{0}\right)-\overline{\varphi}_{t}\left(\mathbf{u}_{0}\right)\right\Vert _{B_{j}}\leq\epsilon C_{j}\left\Vert \mathbf{u}_{0}\right\Vert _{B_{j+1}},\,\,\,\,\,0\leq t\leq1.\label{eq:asymptotic approximation error, proof}
\end{equation}

\item The operator $\overline{\varphi}_{\Delta T}\left(\cdot\right)$ satisfies
\begin{equation}
\left\Vert \overline{\varphi}_{\Delta T}\left(\mathbf{u}_{1}\right)-\overline{\varphi}_{\Delta T}\left(\mathbf{u}_{2}\right)\right\Vert _{B_{j}}\leq\left(1+C_{j}\Delta T\right)\left\Vert \mathbf{u}_{1}-\mathbf{u}_{2}\right\Vert _{B_{j}},\label{eq:condition for proof, coarse solver}
\end{equation}
and its numerical approximation $\widetilde{\varphi}_{\Delta T}\left(\cdot\right)$
satisfies 
\begin{equation}
\left\Vert \widetilde{\varphi}_{\Delta T}\left(\mathbf{u}_{1}\right)-\widetilde{\varphi}_{\Delta T}\left(\mathbf{u}_{2}\right)\right\Vert _{B_{j}}\leq\left(1+C_{j}\Delta T\right)\left\Vert \mathbf{u}_{1}-\mathbf{u}_{2}\right\Vert _{B_{j}}.\label{eq:conditions for proof, coarse solver, 2}
\end{equation}

\item The operators $\mathcal{E}_{\varphi,\overline{\varphi}}\left(\cdot\right)$,
and $\mathcal{E}_{\overline{\varphi},\widetilde{\varphi}}\left(\cdot\right)$
satisfy 
\begin{equation}
\left\Vert \mathcal{E}_{\varphi,\overline{\varphi}}\left(\mathbf{u}_{1}\right)-\mathcal{E}_{\varphi,\overline{\varphi}}\left(\mathbf{u}_{2}\right)\right\Vert _{B_{j}}\leq\epsilon C_{j}\left\Vert \mathbf{u}_{1}-\mathbf{u}_{2}\right\Vert _{B_{j+1}},\label{eq:condition for proof, diff of coarse and fine solver}
\end{equation}
and
\begin{equation}
\left\Vert \mathcal{E}_{\overline{\varphi},\widetilde{\varphi}}\left(\mathbf{u}_{1}\right)-\mathcal{E}_{\overline{\varphi},\widetilde{\varphi}}\left(\mathbf{u}_{2}\right)\right\Vert \leq C_{j}\Delta T^{p+1}\left\Vert \mathbf{u}_{1}-\mathbf{u}_{2}\right\Vert _{B_{j+1}}.\label{eq:condition for proof, diff of coarse and fine solver, 2}
\end{equation}
for some $p\geq1$.
\end{enumerate}
Recalling that $T_{n}=n\Delta T$, a slight modification of the proof
of Theorem~$1$ in \cite{Bal2005} immediately yields the convergence
result stated below. In particular, taking $\Delta T=\epsilon^{1/2}$,
the error scales asymptotically like $\epsilon^{k+1/2}$ after the
$k$th iteration (the choice of scaling $\Delta T\sim\epsilon^{1/2}$
yields a speedup that scales like $\epsilon^{-1/2}$). 
\begin{thm}
\label{thm:convergence}Assuming that $\mathbf{u}_{0}=\mathbf{u}\left(T_{0}\right)\in B_{j+k+1}$,
the error, $\mathbf{u}\left(T_{n}\right)-\mathbf{U}_{n}^{k}$, after
the $k$th parareal iteration is bounded by 
\[
\left\Vert \mathbf{u}\left(T_{n}\right)-\mathbf{U}_{n}^{k}\right\Vert _{B_{j}}\leq C_{k,j}\left(\Delta T^{p}+\epsilon\right)\left(\Delta T^{p}+\frac{\epsilon}{\Delta T}\right)^{k}\left\Vert \mathbf{u}_{0}\right\Vert _{B_{k+j+1}},
\]
where $C_{k,j}$ is a constant that depends only on the constants
$C_{m}$, $m=0,1,\ldots,k+j$.\end{thm}
\begin{proof}
The proof is by induction on $k$. When $k=0$,
\begin{eqnarray*}
\left\Vert \mathbf{u}\left(T_{n}\right)-\mathbf{U}_{n}^{0}\right\Vert _{B_{j}} & = & \left\Vert \varphi_{n\Delta T}\left(\mathbf{u}_{0}\right)-\widetilde{\varphi}_{n\Delta T}\left(\mathbf{u}_{0}\right)\right\Vert _{B_{j}}\\
 & \leq & \left\Vert \varphi_{n\Delta T}\left(\mathbf{u}_{0}\right)-\overline{\varphi}_{n\Delta T}\left(\mathbf{u}_{0}\right)\right\Vert _{B_{j}}+\left\Vert \overline{\varphi}_{n\Delta T}\left(\mathbf{u}_{0}\right)-\widetilde{\varphi}_{n\Delta T}\left(\mathbf{u}_{0}\right)\right\Vert _{B_{j}}\\
 & \leq & B_{0,j}\left(\epsilon+\Delta T^{p}\right)\left\Vert \mathbf{u}_{0}\right\Vert _{B_{j+1}},
\end{eqnarray*}
where, in the last inequality, we used (\ref{eq:asymptotic approximation error, proof})
to bound the first term and a classical result (see e.g. \cite{Bal2005}
for more discussion) to bound the second term.

Now assume that
\[
\left\Vert \mathbf{u}\left(T_{n}\right)-\mathbf{U}_{n}^{k-1}\right\Vert _{B_{j}}\leq C_{k-1,j}\left(\Delta T+\epsilon\right)\left(\Delta T^{p}+\frac{\epsilon}{\Delta T}\right)^{k-1}\left\Vert \mathbf{u}_{0}\right\Vert _{B_{j+k}}.
\]
holds. Using (\ref{eq:parareal iteration for proof}), $\mathbf{u}\left(T_{n}\right)=\varphi_{\Delta T}\left(\mathbf{u}\left(T_{n-1}\right)\right)$,
and the definitions of $\mathcal{E}_{\varphi,\overline{\varphi}}\left(\cdot\right)$
and $\mathcal{E}_{\overline{\varphi},\widetilde{\varphi}}\left(\cdot\right)$,
we rewrite the difference $\mathbf{u}\left(T_{n}\right)-\mathbf{U}_{n}^{k}$
in the form
\begin{eqnarray*}
\mathbf{u}\left(T_{n}\right)-\mathbf{U}_{n}^{k} & = & \left(\widetilde{\varphi}_{\Delta T}\left(\mathbf{u}\left(T_{n-1}\right)\right)-\widetilde{\varphi}_{\Delta T}\left(\mathbf{U}_{n-1}^{k}\right)\right)+\\
 &  & \,\,\,\left(\mathcal{E}_{\varphi,\overline{\varphi}}\left(\mathbf{u}\left(T_{n-1}\right)\right)-\mathcal{E}_{\varphi,\overline{\varphi}}\left(\mathbf{U}_{n-1}^{k-1}\right)\right)+\left(\mathcal{E}_{\overline{\varphi},\widetilde{\varphi}}\left(\mathbf{u}\left(T_{n-1}\right)\right)-\mathcal{E}_{\overline{\varphi},\widetilde{\varphi}}\left(\mathbf{U}_{n-1}^{k-1}\right)\right).
\end{eqnarray*}
Using (\ref{eq:conditions for proof, coarse solver, 2}) for the first
term and (\ref{eq:condition for proof, diff of coarse and fine solver})-(\ref{eq:condition for proof, diff of coarse and fine solver, 2})
for the second and third terms, we obtain that
\begin{eqnarray*}
\left\Vert \mathbf{u}\left(T_{n}\right)-\mathbf{U}_{n}^{k}\right\Vert _{B_{j}} & \leq & \left(1+C_{j}\Delta T\right)\left\Vert \mathbf{u}\left(T_{n-1}\right)-\mathbf{U}_{n-1}^{k}\right\Vert _{B_{j}}+\\
 &  & \,\,\,\,\,\,\,\, C_{j}\left(\Delta T^{p+1}+\epsilon\right)\left\Vert \mathbf{u}\left(T_{n-1}\right)-\mathbf{U}_{n-1}^{k-1}\right\Vert _{B_{j+1}}\\
 & \leq & \left(1+C_{j}\Delta T\right)\left\Vert \mathbf{u}\left(T_{n-1}\right)-\mathbf{U}_{n-1}^{k}\right\Vert _{B_{j}}+\\
 &  & \,\,\,\,\,\,\,\Delta TC_{j}C_{k-1,j+1}\left(\Delta T^{p}+\frac{\epsilon}{\Delta T}\right)^{k+1}\left\Vert \mathbf{u}_{0}\right\Vert _{B_{j+k+1}},
\end{eqnarray*}
where we used the induction hypothesis in the previous step. Finally,
applying the discrete Gronwall inequality, we obtain that
\begin{eqnarray*}
\left\Vert \mathbf{u}\left(T_{n}\right)-\mathbf{U}_{n}^{k}\right\Vert _{B_{j}} & \leq & \left(e^{C_{j}\left(T_{n}-T_{0}\right)}-1\right)C_{k-1,j+1}\left(\Delta T^{p}+\frac{\epsilon}{\Delta T}\right)^{k+1}\left\Vert \mathbf{u}_{0}\right\Vert _{B_{j+k+1}}\\
 & \leq & C_{k,j}\left(\Delta T^{p}+\frac{\epsilon}{\Delta T}\right)^{k+1}\left\Vert \mathbf{u}_{0}\right\Vert _{B_{j+k+1}},
\end{eqnarray*}
where
\[
C_{k,j}=\left(e^{C_{j}}-1\right)C_{k-1,j+1}.
\]

\end{proof}
~~~

The constant $C_{k,j}$ in Theorem~\ref{thm:convergence} can potentially
grow with increasing $k$, and therefore convergence is to be understood
in an asymptotic sense (that is, fixed $k$ and decresing $\epsilon$). 

A straightforward modification of the proof in \cite{Gander} yields
a similar bound as in Theorem~\ref{thm:convergence}, but with a
constant that decreases with increasing $k$, and, in fact, yields
superlinear convergence. In particular, this result demonstrates convergence
for fixed $\epsilon$, as $k$ increases. Although the error bound
is therefore more powerful, its derivation also requires more stringent
assumptions (in particular, these assumptions may not hold in the
infinite-dimensional setting).

To state this result, we suppose that the Banach spaces $B_{j}$ coincide,
$B_{0}=B_{1}=\ldots$, and that bounds of the form (\ref{eq:stability, proof})-(\ref{eq:condition for proof, diff of coarse and fine solver, 2})
hold for a fixed constant $C=C_{j}$, $j=0,1,\ldots$ Then we have
the following error bounds. We also assume that the number $N$ of
large time steps $\Delta T$ taken satisfies $N=\epsilon^{-1/2}$,
and therefore $\Delta T=\epsilon^{1/2}$ (based on the above complexity
analysis, this yields, in principle, optimal parallel speedup).
\begin{thm}
\label{thm:convergence-1}The error, $\mathbf{u}\left(T_{n}\right)-\mathbf{U}_{n}^{k}$,
after the $k$th parareal iteration is bounded by 
\[
\left\Vert \mathbf{u}\left(T_{n}\right)-\mathbf{U}_{n}^{k}\right\Vert _{B_{0}}\leq\epsilon^{k/2}\left(\epsilon+\epsilon^{p/2}\right)\frac{\left(2C\right)^{k}}{\left(k+1\right)!}e^{C\left(T_{n}-T_{k+1}\right)}\left(e^{C\left(T_{n}-T_{0}\right)}-1\right).
\]
\end{thm}
\begin{proof}
As in the proof of Theorem~\ref{thm:convergence},
\begin{eqnarray*}
\left\Vert \mathbf{u}\left(T_{n}\right)-\mathbf{U}_{n}^{k}\right\Vert _{B_{0}} & \leq & \left(1+C\Delta T\right)\left\Vert \mathbf{u}\left(T_{n-1}\right)-\mathbf{U}_{n-1}^{k}\right\Vert _{B_{0}}+\\
 &  & \,\,\,\,\,\,\,\, C\left(\Delta T^{p+1}+\epsilon\right)\left\Vert \mathbf{u}\left(T_{n-1}\right)-\mathbf{U}_{n-1}^{k-1}\right\Vert _{B_{0}},
\end{eqnarray*}
and 
\[
\left\Vert \mathbf{u}\left(T_{n}\right)-\mathbf{U}_{n}^{0}\right\Vert _{B_{0}}\leq\widetilde{C}\left(\epsilon+\Delta T^{p}\right)\left\Vert \mathbf{u}_{0}\right\Vert _{B_{0}},
\]
for some constant $\widetilde{C}$. In fact, it can be shown that
$\widetilde{C}=\left(e^{C\left(T_{n}-T_{0}\right)}-1\right)$ suffices.
Now following the proof of Theorem~$1$ in \cite{Gander}, we obtain
the bound
\begin{eqnarray*}
\left\Vert \mathbf{u}\left(T_{n}\right)-\mathbf{U}_{n}^{k}\right\Vert _{B_{0}} & \leq & \widetilde{C}\left(\epsilon+\Delta T^{p}\right)\frac{\left(C\left(\epsilon+\Delta T^{p+1}\right)\right)^{k}}{\left(k+1\right)!}\left(1+C\Delta T\right)^{n-k-1}n^{k}\\
 & \leq & \widetilde{C}\left(\epsilon+\Delta T^{p}\right)\frac{\left(Cn\epsilon+Cn\epsilon\right)^{k}}{\left(k+1\right)!}e^{C\left(T_{n}-T_{k+1}\right)}\\
 & \leq & \epsilon^{k/2}\left(\epsilon+\epsilon^{p/2}\right)\widetilde{C}\frac{\left(2C\right)^{k}}{\left(k+1\right)!}e^{C\left(T_{n}-T_{k+1}\right)}.
\end{eqnarray*}
where the first inequality used that $\Delta T^{p+1}\leq\epsilon$,
the second inequality used that $1+x\leq e^{x}$, and the third inequality
used that $n\epsilon\leq N\epsilon\leq\epsilon^{1/2}$ (recall the
number $N$ of big time steps is given by $N=\epsilon^{-1/2}$). 
\end{proof}

\section{Numerical examples\label{sec:Numerical-examples}}

\subsection{Rotating Shallow Water Equations\label{sub:Rotating-Shallow-Water}}

We consider as a test problem the non-dimensional rotating shallow
water equations (RSW) equations,

\begin{eqnarray}
\frac{\partial v_{1}}{\partial t}+\frac{1}{\epsilon}\left(-v_{2}+F^{-1/2}\frac{\partial h}{\partial x}\right)+v_{1}\frac{\partial v_{1}}{\partial x} & = & \mu\partial_{x}^{4}v_{1},\nonumber \\
\frac{\partial v_{2}}{\partial t}+\frac{1}{\epsilon}v_{1}+v_{1}\frac{\partial v_{2}}{\partial x} & = & \mu\partial_{x}^{4}v_{2},\label{eq:RSW equations, explicit]}\\
\frac{\partial h}{\partial t}+\frac{F^{-1/2}}{\epsilon}\frac{\partial v_{1}}{\partial x}+\frac{\partial}{\partial x}\left(hv_{1}\right) & = & \mu\partial_{x}^{4}h,\nonumber 
\end{eqnarray}
with spatially periodic boundary conditions on the interval $\left[0,2\pi\right]$.
Here $h\left(x,t\right)$ denotes the surface height of the fluid,
and $v_{1}\left(x,t\right)$ and $v_{2}\left(x,t\right)$ denote the
horizontal fluid velocities. The non-dimensional parameter $\epsilon$
denotes the Rossby number (a ratio of the characteristic advection
time to the rotation time), and is often small (e.g. $10^{-2}$) in
realistic oceanic flows. The non-dimensional parameter $F^{1/2}\epsilon$
gives the Froude number (a ratio of the characteristic fluid velocity
to the gravity wave speed), where $F=\mathcal{O}\left(1\right)$ is
a free parameter which we set to unity in our subsequent calculations.
The scaling we have taken is for quasi-geostrophic dynamics (cf. \cite{Majda1998}),
which governs the fluid flow dominated by strong stratification and
strong constant rotation. As is standard, we also introduce a hyperviscosity
operator $\mu\partial_{x}^{4}$ to prevent singularities from forming;
using hyperviscosity also ensures that the low frequencies are less
effected by dissipation than the high frequencies. The RSW equations
represent a standard framework in which to develop and test new numerical
algorithms for geophysical fluid applications.

To relate equations (\ref{eq:RSW equations, explicit]}) to the abstract
formulation (\ref{eq:full eqn}), we define 
\[
\mathbf{u}\left(t,\mathbf{x}\right)=\left(\begin{array}{c}
v_{1}\left(t,x\right)\\
v_{2}\left(t,x\right)\\
h\left(t,x\right)
\end{array}\right).
\]
Then the system (\ref{eq:RSW equations, explicit]}) can be written
in the form (\ref{eq:full eqn}) by setting

\[
\mathcal{L}=\left(\begin{array}{ccc}
0 & -1 & F^{-1/2}\partial_{x}\\
1 & 0 & 0\\
F^{-1/2}\partial_{x} & 0 & 0
\end{array}\right),\,\,\,\mathcal{D}=\mu\partial_{x}^{4}\left(\begin{array}{ccc}
1 & 0 & 0\\
0 & 1 & 0\\
0 & 0 & 1
\end{array}\right),\,\,\,\mathcal{N}\left(\mathbf{u}\right)=\left(\begin{array}{c}
v_{1}\left(v_{1}\right)_{x}\\
v_{1}\left(v_{2}\right)_{x}\\
\left(hv_{1}\right)_{x}
\end{array}\right).
\]
Because we work in a periodic domain, it is also convenient to consider
(\ref{eq:full eqn}) and (\ref{eq:averaged equation}) in the Fourier
domain. This will also make explicit the time-averaging in (\ref{eq:averaged equation})
(periodicity is not required for our approach, and is only used to
simplify the numerical scheme). A straightforward calculation (see
\cite{Majda1998}) shows that 
\[
\mathcal{L}\left(e^{ikx}\mathbf{r}_{k}^{\alpha}\right)=i\omega_{k}e^{ikx}\mathbf{r}_{k}^{\alpha},\,\,\,\,\omega_{k}^{\alpha}=\alpha\sqrt{1+F^{-1}k^{2}},
\]
where $\mathbf{r}_{k}^{\alpha}$ is a vector that depends on the wavenumber
$k$ and $\alpha=-1,0,1$. Therefore, by expanding the function $\overline{\mathbf{u}}\left(t\right)$
in the basis of eigenfunctions for $\mathcal{L}$, we have that 
\[
e^{\tau\mathcal{L}}\overline{\mathbf{u}}\left(t\right)=\sum_{k\in\mathbb{Z}}e^{ikx}\sum_{\alpha=-1}^{1}e^{i\omega_{k}^{\alpha}}u_{k}^{\alpha}\mathbf{r}_{k}^{\alpha}.
\]

As shown in e.g. \cite{MAJDA:2003}, the nonlinear term $e^{s\mathcal{L}}\mathcal{N}\left(e^{-s\mathcal{L}}\overline{\mathbf{u}}\left(t\right)\right)$
in (\ref{eq:full eqn}) can be written in the form

\begin{equation}
\sum_{k}e^{ikx}\left(\sum_{k_{1}+k_{2}=k}\sum_{\alpha_{1},\alpha_{2}}e^{i\left(\omega_{k}^{\alpha}-\omega_{k_{1}}^{\alpha_{1}}-\omega_{k_{2}}^{\alpha_{2}}\right)s}C_{k,k_{1},k_{2}}^{\alpha,\alpha_{1},\alpha_{2}}u_{k_{1}}^{\alpha_{1}}\left(t\right)u_{k_{2}}^{\alpha_{2}}\left(t\right)\right)\mathbf{r}_{k}^{\alpha}.\label{eq:nonlinear term for average, RSW}
\end{equation}
The interaction coefficients $C_{k,k_{1},k_{2}}^{\alpha,\alpha_{1},\alpha_{2}}$
are explicitly given in e.g. \cite{MAJDA:2003}. From (\ref{eq:nonlinear term for average, RSW})
it is clear that the time average defining $\overline{\mathcal{N}}\left(\overline{\mathbf{u}}\left(t\right)\right)$
in (\ref{eq:average of nonlinear term}) retains only three-wave resonances.
Indeed, since 
\[
\lim_{\text{T}\rightarrow\infty}\frac{1}{T}\int_{0}^{T}e^{i\left(\omega_{k}^{\alpha}-\omega_{k_{1}}^{\alpha_{1}}-\omega_{k_{2}}^{\alpha_{2}}\right)s}ds=\begin{cases}
0, & \,\,\,\,\omega_{k}^{\alpha}-\omega_{k_{1}}^{\alpha_{1}}-\omega_{k_{2}}^{\alpha_{2}}\neq0\\
1, & \,\,\,\,\omega_{k}^{\alpha}-\omega_{k_{1}}^{\alpha_{1}}-\omega_{k_{2}}^{\alpha_{2}}=0,
\end{cases}
\]
we see that the time average is given by 
\begin{equation}
\overline{\mathcal{N}}\left(\overline{\mathbf{u}}\left(t\right)\right)=\sum_{k}e^{ikx}\left(\sum_{\left(k,k_{1},k_{2},\alpha,\alpha_{1},\alpha_{2}\right)\in S_{\text{r}}}C_{k,k_{1},k_{2}}^{\alpha,\alpha_{1},\alpha_{2}}u_{k_{1}}^{\alpha_{1}}\left(t\right)u_{k_{2}}^{\alpha_{2}}\left(t\right)\right)\mathbf{r}_{k}^{\alpha}.\label{explicit average, NL term}
\end{equation}
Here the resonant set $S_{\text{r}}$ is defined by 
\[
S_{\text{r}}=\left\{ \left(k,k_{1},k_{2},\alpha,\alpha_{1},\alpha_{2}\right)\mid k_{1}+k_{2}=k\,\,\,\,\text{and}\,\,\,\,\omega_{k}^{\alpha}-\omega_{k_{1}}^{\alpha_{1}}-\omega_{k_{2}}^{\alpha_{2}}=0\right\} .
\]
The HMM (outlined in Section~\ref{sec:An-asymptotic-slow}) allows
the resonant terms in (\ref{explicit average, NL term}) to be efficiently
computed.

\subsection{Numerical experiments on the RSW equations\label{sub:Numerical-experiments-on}}

We solve equations (\ref{eq:RSW equations, explicit]}) with the initial
conditions consisting of $v_{1}\left(0,x\right)=v_{2}\left(0,x\right)=0$
and 
\begin{equation}
h\left(x,0\right)=c_{1}\left(e^{-4\left(x-\pi/2\right)^{2}}\sin\left(3\left(x-\pi/2\right)\right)+e^{-2\left(x-\pi\right)^{2}}\sin\left(8\left(x-\pi\right)\right)\right)+c_{0},\label{eq:initial condition 1-1}
\end{equation}
where $c_{1}$ and $c_{2}$ are chosen so that 
\[
\int_{0}^{2\pi}h\left(0,x\right)dx=0,\,\,\,\,\max_{x}\left|h\left(x,0\right)\right|=1.
\]
We also choose the viscosity parameter $\mu=10^{-4}$, and the values
$\epsilon=10^{-2}$, $\epsilon=10^{-1}$, and $\epsilon=1$ (corresponding
to strong, weak, and no scale separation), which are physically realistic
values in geophysical ocean flows; smaller values of $\epsilon$ would
yield even greater parallel speedup, but are less physically relevant.
Although the choice of the initial height $h\left(x,0\right)$ in
(\ref{eq:initial condition 1-1}) is somewhat arbitrary, the conclusions
in this section appear to be insensitive to the initial conditions
(as long as they are sufficiently smooth). For all three choices of
$\epsilon$, we perform the following numerical experiments. First,
we compute the solution using the asymptotic parareal method outlined
in Section~\ref{sec:Parareal-Method}. We then compare the estimated
parallel speedup with the results of using three numerical integrators
in serial: exponential time differencing $4$th order Runge-Kutta
(ETDRK4) \cite{CoxMathews:2002}, the operator integrating factor
method (cf. \cite{KassamTrefethen:2005}), and Strang splitting (see
Algorithm~\ref{alg:Fine-solver-Strang} of Section~\ref{sec:Parareal-Method}).
Finally, for comparison we compute the solution using the standard
parareal method by solving the full equation (\ref{eq:full eqn})
using big and small step sizes $\Delta T$ and $\Delta t$, where
we use Strang splitting for the coarse solver (we find that using
ETDRK4 or OIFS as a coarse solver yields similar parallel speedup).
For the sake of simplicity, we also assume that the cost of computing
a single time step using ETDRK4, OIFS, and Strang splitting is the
same; a more careful analysis that takes into account the number of
operations required for each integrator would yield greater parallel
speedup.

As we show below, with the coarse time step $\Delta T\geq3/10$, we
find similar convergence and accuracy in the asymptotic parareal method
for all values of $\epsilon$. The parallel speedup for $\epsilon=10^{-2}$
is about a factor of $100$ relative to using ETDRK4, the integrating
factor method, or Strang splitting in serial. In contrast, the standard
parareal method (using ETDKRK4, the integrating factor method, or
Strang splitting as a coarse solver) requires much smaller time steps,
resulting in a parallel speedup that is about $10$ times smaller
for $\epsilon=10^{-2}$. In our estimates of the speedup, we assume
that the time average (\ref{eq:numerical average, discretized}) is
computed in parallel, and that the costs of computing the coarse and fine time steps
are the same. 

In our first numerical experiment, we take $\epsilon=10^{-2}$. For
the asymptotic parareal method, we use a coarse time step $\Delta T=50\times\epsilon=1/2$,
a fine time step $\Delta t=\epsilon/25=1/2500$, and $N=\Delta T/\Delta t=1250$
time intervals $\left[\left(n-1\right)\Delta T,n\Delta T\right]$.%
{} We use $T_{0}\left(\epsilon\right)=150\times\epsilon=3$ and $\overline{M}=450$
terms in the time average (\ref{eq:numerical average, discretized})
for the coarse solver. Even though the solution executes many temporal
oscillations within the time intervals $\left[n\Delta T,\left(n+1\right)\Delta T\right]$---see
Figure~\ref{fig:4th fourier coeff} for a plot of the $4$th Fourier
coefficient $\hat{h}\left(k,t\right)$ for $0\leq t\leq\Delta T$---the
method converges in a small number of iterations. In fact, Figure~\ref{fig:errors parareal HMM, eps 2}
shows the maximum relative $L^{\infty}$ error,
\[
\max_{0\leq n\leq N}\left\Vert \mathbf{U}_{n}^{k}-\mathbf{u}\left(n\Delta T,\cdot\right)\right\Vert _{\infty}/\left\Vert \mathbf{u}\left(n\Delta T,\cdot\right)\right\Vert _{\infty},
\]
as a function of the iteration levels $k=0,1,\ldots,5$. Comparing
the parallel speedup relative to directly integrating the full equation
(\ref{eq:full eqn}) using ETDRK4, the integrating factor method,
and Strang splitting, we find that we need step sizes $\Delta t=\epsilon/25$,
$\Delta t=\epsilon/20$, and $\Delta t=\epsilon/20$, respectively,
in order to obtain a comparable accuracy. Therefore, using the complexity
analysis in Section~\ref{sec:Error-and-complexity}, we expect a
parallel speedup of 
\[
\frac{N\left(\Delta T/\Delta t\right)}{5\left(\left(\Delta T/\Delta t\right)+N\right)+N}\approx110.
\]

To constrast this speedup with that obtained using a standard version
of the parareal method, we show in Figure~\ref{fig:errors parareal HMM, eps 2}
the results obtained via a coarse solver based on solving the full
equation (\ref{eq:full eqn}) using Strang splitting (the results
are similar with using ETDRK4 or the integrating factor method as
coarse solvers). In this experiment, we take $\Delta t=\epsilon/25=1/2500$,
$\Delta T=4\epsilon=1/25$, and $N=\Delta T/\Delta t=100$, which
yields an expected parallel speedup that is $10$ times smaller.

\begin{figure}
\caption{Plot of the $4$th Fourier coefficient $\hat{h}\left(k,t\right)$
as a function of time $0\leq t\leq\Delta T$, for $\epsilon=10^{-2}$\label{fig:4th fourier coeff}}

\includegraphics[scale=0.4]{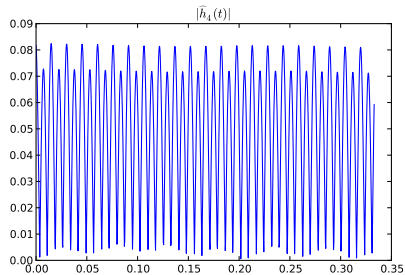}
\end{figure}

\begin{figure}
\caption{Maximum relative $L^{\infty}$ error (on a$\log_{10}$ scale), $\max_{0\leq n\leq N}\left\Vert \mathbf{U}_{n}^{k}-\mathbf{u}\left(n\Delta T,\cdot\right)\right\Vert _{\infty}/\left\Vert \mathbf{u}\left(n\Delta T,\cdot\right)\right\Vert _{\infty}$,
as a function of the iteration level $k$; the initial condition (\ref{eq:initial condition 1-1})
is used, and $\epsilon=10^{-2}$. The solid green line depicts the
errors from the asymptotic parareal method (Parareal-HMM) with a coarse
time step $\Delta T=50\epsilon$, and the dashed and dashed-dotted
lines depict the errors from the standard parareal method using
Strang splitting (Parareal-Strang) with $\Delta T=4\epsilon$ and
$\Delta T=5\epsilon,$ respectively. A fine time step $\Delta t=\epsilon/25$
is used for all three cases. \label{fig:errors parareal HMM, eps 2}}

\includegraphics[scale=0.5]{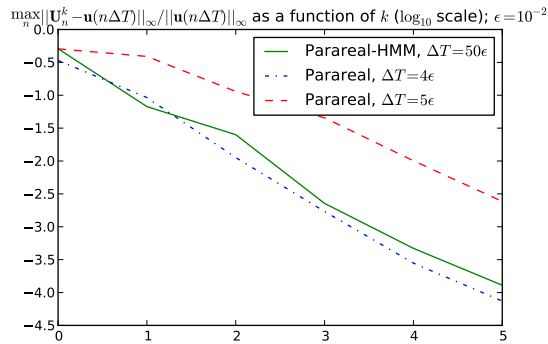}
\end{figure}

\begin{figure}
\caption{Maximum relative $L^{\infty}$ error (on a$\log_{10}$ scale), $\max_{0\leq n\leq N}\left\Vert \mathbf{U}_{n}^{k}-\mathbf{u}\left(n\Delta T,\cdot\right)\right\Vert _{\infty}/\left\Vert \mathbf{u}\left(n\Delta T,\cdot\right)\right\Vert _{\infty}$,
as a function of the iteration level $k$; the initial condition (\ref{eq:initial condition 1-1})
is used, and $\epsilon=10^{-1}$. The solid line depicts the
errors from the asymptotic parareal method (Parareal-HMM) with a coarse
time step $\Delta T=3\epsilon$, and the dashed and dashed-dotted
lines depict the errors from the standard parareal method using
Strang splitting (Parareal-Strang) with $\Delta T=\epsilon$ and $\Delta T=2\epsilon,$
respectively. A fine time step $\Delta t=\epsilon/50$ is used for
all three cases.\label{fig:errors parareal HMM, eps 1}}

\includegraphics[scale=0.5]{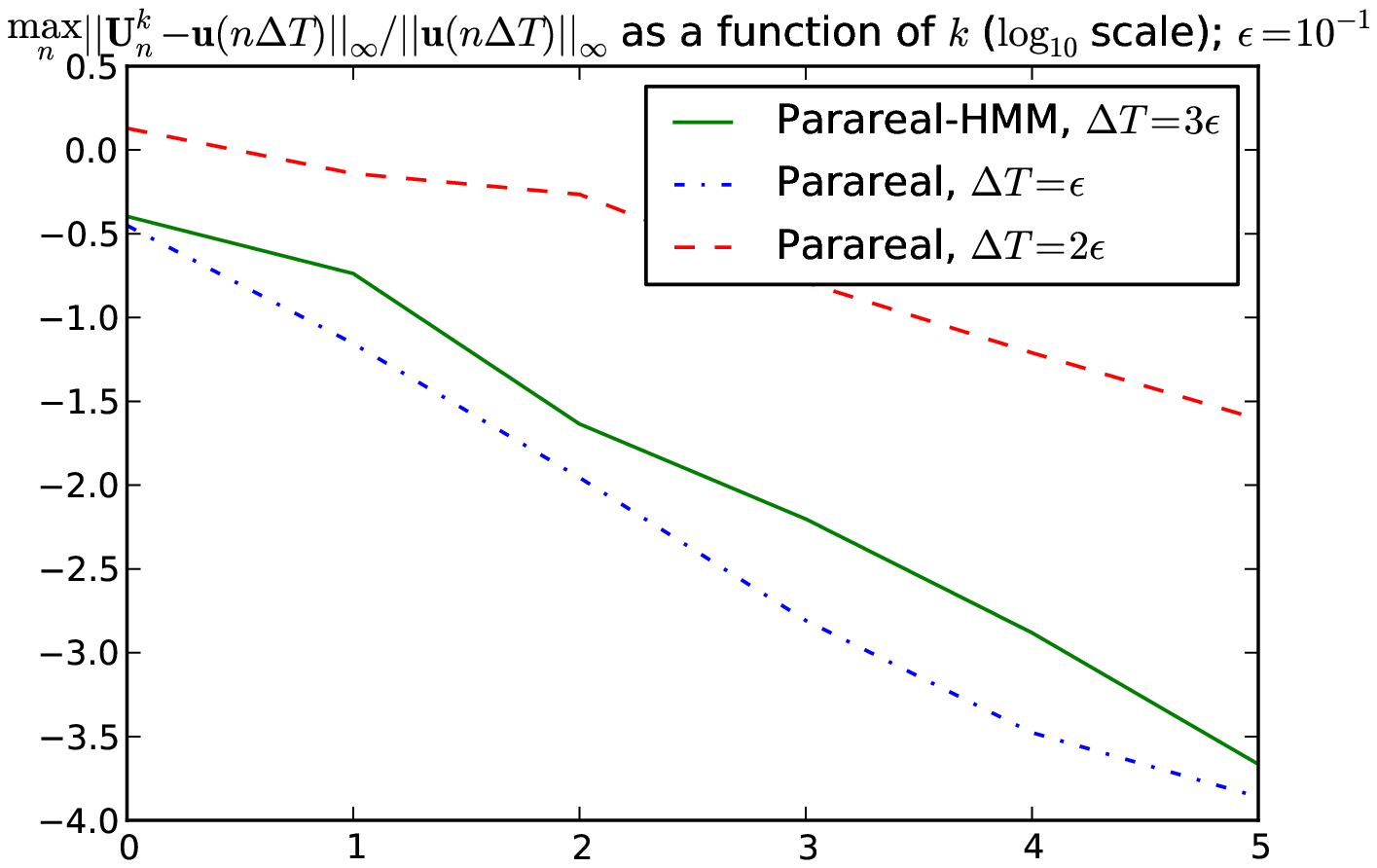}
\end{figure}

\begin{figure}
\caption{Maximum relative $L^{\infty}$ error (on a$\log_{10}$ scale), $\max_{0\leq n\leq N}\left\Vert \mathbf{U}_{n}^{k}-\mathbf{u}\left(n\Delta T,\cdot\right)\right\Vert _{\infty}/\left\Vert \mathbf{u}\left(n\Delta T,\cdot\right)\right\Vert _{\infty}$,
as a function of the iteration level $k$; the initial condition (\ref{eq:initial condition 1-1})
is used, and $\epsilon=1$. The solid line depicts the errors
from the asymptotic parareal method (Parareal-HMM) with a coarse time
step $\Delta T=3/10\epsilon$, and the dashed and dashed-dotted
lines depict the errors from the standard parareal method using
Strang splitting (Parareal-Strang) with $\Delta T=3/10\epsilon$ and
$\Delta T=4/10\epsilon,$ respectively. A fine time step $\Delta t=\epsilon/200$
is used for all three cases.\label{fig:errors parareal HMM, eps 0}}

\includegraphics[scale=0.5]{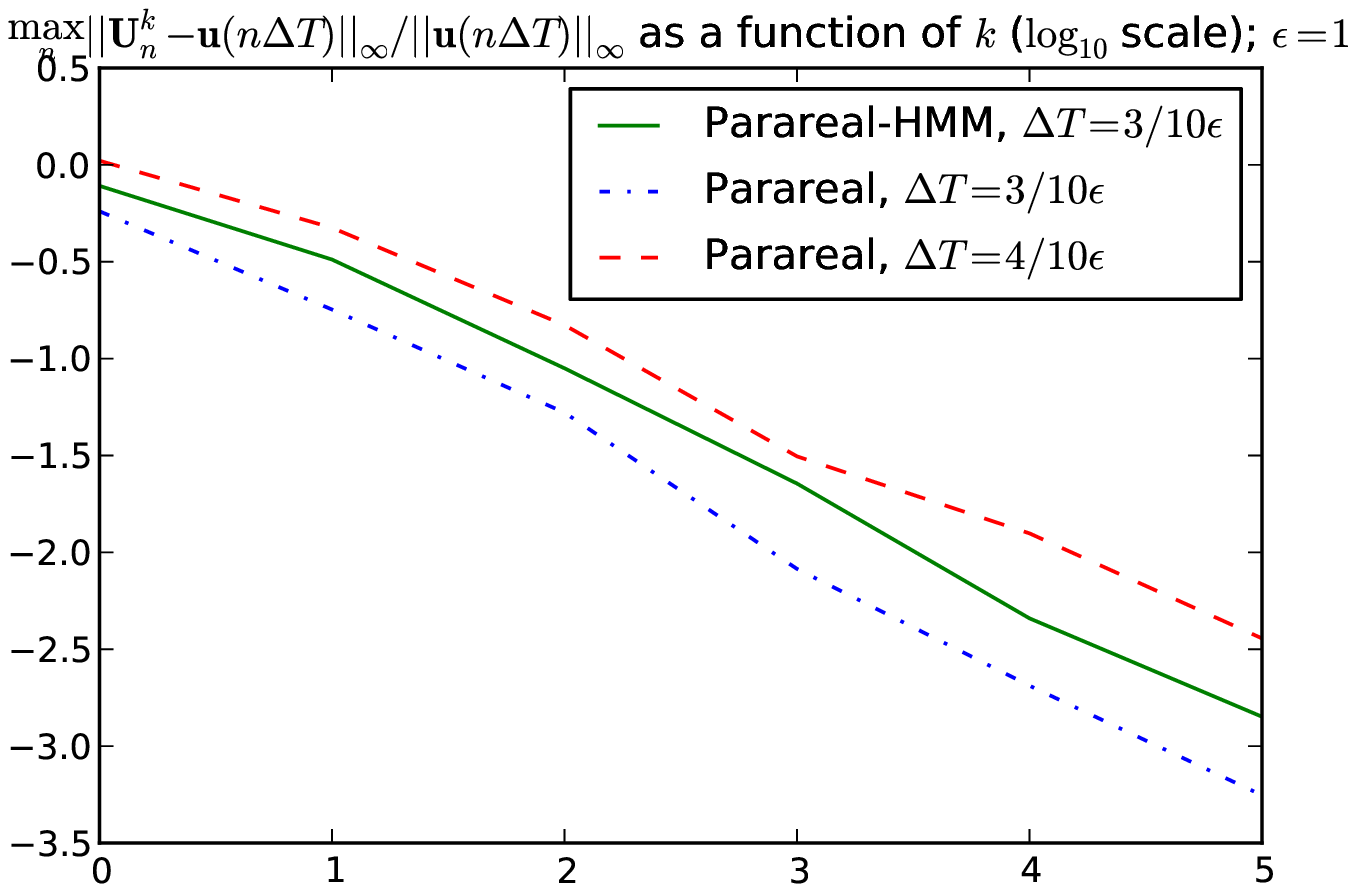}
\end{figure}

In our second experiment, we take $\epsilon=10^{-1}$, corresponding
to weak scale separation. In Figure~\ref{fig:errors parareal HMM, eps 1},
we show the maximum relative $L^{\infty}$ error as a function of
the iteration levels $k=0,1,\ldots,5$.%
. Here we used a coarse time step $\Delta T=3\epsilon=3/10$, a fine
time step $\Delta t=\epsilon/25=1/250$, and $N=150$ time intervals
$\left[\left(n-1\right)\Delta T,n\Delta T\right]$. We also take $T_{0}\left(\epsilon\right)=5\times\epsilon=1/2$
and $\overline{M}=10$. Comparing the parallel speedup relative to
directly integrating the full equation (\ref{eq:full eqn}) using
ETDRK4, integrating factor method, and Strang splitting, we find that
we need time steps $\Delta t=\epsilon/20=1/200$, $\Delta t=\epsilon/10=1/100$,
and $\Delta t=\epsilon/20=1/200$, respectively, in order to obtain
a comparable accuracy. Thus, we obtain an estimated parallel speed
of about $1/6\left(\Delta T/\Delta t\right)=13$ relative to using
Strang splitting in serial; taking into account the number of operations
required for each time step in ETDRK4 and the integrating factor method,
the estimated parallel speedup relative to these integrators will
be comparable. In contrast, Figure~\ref{fig:errors parareal HMM, eps 1}
also shows the relative $L^{\infty}$ errors, where this time the
standard parareal algorithm is used with Strang splitting for the
coarse and fine solvers, and using the same step sizes $\Delta T=\epsilon=1/20$
and $\Delta t=\epsilon/20=1/200$ (again, similar speedup is obtained
with ETDRK4 and the integrating factor method). Here the estimated
parallel speedup is $3$ times smaller.

In our final experiment, we take $\epsilon=1$, corresponding to no
scale separation. In Figure~\ref{fig:errors parareal HMM, eps 0},
we show the relative $L^{\infty}$ error for the asymptotic parareal
method, where we use a coarse time step $\Delta T=3/10$, a fine time
step $\Delta t=1/200$, and $N=60$ time intervals $\left[\left(n-1\right)\Delta T,n\Delta T\right]$.
We also take $T_{0}\left(\epsilon\right)=3/10$ and $\overline{M}=10$.
Comparing the parallel speedup relative to directly integrating the
full equation (\ref{eq:full eqn}), we obtain an estimated parallel
speedup of about a factor of $5$; the speedup with the standard parareal
method (and using Strang splitting for a coarse solver) is about the
same.

\section{Summary\label{sub:summary}}

In this paper we have introduced an asymptotic-parallel-in-time method
for solving highly oscillatory PDEs. The method is a modification
of the parareal algorithm introduced by \cite{LI-MA-TU}. The modification
replaces the coarse solver used in \cite{LI-MA-TU} by a numerically
computed locally asymtotic solution based on the asymptotic mathematical
structure of the equations (\cite{schochet1994}, \cite{MAJDA:2003},
\cite{Majda1998}, \cite{Wingate2011}) and concepts used in HMM.

In addition to presenting the method we also include pseudocode.
We discuss the performance of the method when $\epsilon=$1, which
is important for using the method in realistic simulations where the
time scale separation may vary in space and time. We also present
a complexity analysis that shows that the parallel speed-up increases
as $\epsilon$ decreases which results in an arbitrarily greater efficiency
gain relative to standard numerical integrators and a Theorem following
\cite{Gander} that shows that as long as the constants remain bounded
the error decreases by a factor of $\epsilon^{1/2}$ after each iteration.

We also present numerical experiments for the shallow water equations.
These results demonstrate that the parallel speedup is more than $100$
relative to exponential integrators such as ETDRK4 (for realistic
parameter values in the shallow water equations); the speedup is also
more than $10$ relative to using the standard parareal method with
a linearly exact coarse solver. Finally, the results demonstrate that
the method works in the absence of scale separation, and with as much
speedup as the standard parareal method.


\begin{thebibliography}{10}

\bibitem{ArielEngquistKimLeeTsai:2013}
G.~Ariel, B.~Engquist, S.~Kim, Y.~Lee, and R.~Tsai.
\newblock A multiscale method for highly oscillatory dynamical systems using a
  poincaré map type technique.
\newblock {\em Journal of Scientific Computing}, 54:247--268, 2013.

\bibitem{ArielEngquistTsai:2009}
Gil Ariel, Bjorn Engquist, and Richard Tsai.
\newblock A multiscale method for highly oscillatory ordinary differential
  equations with resonance.
\newblock {\em Math. Comp.}, 78(266):929--956, 2009.

\bibitem{ArielEngquistBjörnKreissTsai:2009}
Gil Ariel, Björn Engquist, Heinz-Otto Kreiss, and Richard Tsai.
\newblock Multiscale computations for highly oscillatory problems.
\newblock In Björn Engquist, Per Lötstedt, and Olof Runborg, editors, {\em
  Multiscale Modeling and Simulation in Science}, volume~66 of {\em Lecture
  Notes in Computational Science and Engineering}, pages 237--287. Springer
  Berlin Heidelberg, 2009.

\bibitem{Baffico2002}
L.~Baffico, S.~Bernard, Y.~Maday, G.~Turinici, and G.~Z\'erah.
\newblock Parallel-in-time molecular-dynamics simulations.
\newblock {\em Phys. Rev. E}, 66:057701, Nov 2002.

\bibitem{Bal2005}
Guillaume Bal.
\newblock On the convergence and the stability of the parareal algorithm to
  solve partial differential equations.
\newblock In {\em Domain Decomposition Methods in Science and Engineering},
  volume~40 of {\em Lecture Notes in Computational Science and Engineering},
  pages 425--432. Springer Berlin Heidelberg, 2005.

\bibitem{CastellaChartierFaou2009}
F.~Castella, P.~Chartier, and E.~Faou.
\newblock An averaging technique for highly oscillatory {H}amiltonian problems.
\newblock {\em SIAM J. Numer. Anal.}, 47(4):2808--2837, 2009.

\bibitem{Charney1948}
J.G. Charney.
\newblock On the scale of atmospheric motions.
\newblock {\em Geofysiske Publikasjoner}, 17(2):3--17, 1948.

\bibitem{Charney1949}
J.G. Charney.
\newblock On a physical basis for numerical prediction of large-scale motions
  in the atmosphere.
\newblock {\em Journal of Meterology}, 6:371--385, 1949.

\bibitem{CoxMathews:2002}
S.M. Cox and P.C. Matthews.
\newblock Exponential time differencing for stiff systems.
\newblock {\em Journal of Computational Physics}, 176(2):430 -- 455, 2002.

\bibitem{Dai}
X.~Dai and Y.~Maday.
\newblock Stable parareal in time method for first- and second-order hyperbolic
  systems.
\newblock {\em SIAM Journal on Scientific Computing}, 35(1):A52--A78, 2013.

\bibitem{Weinan:2003}
Weinan E.
\newblock Analysis of the heterogeneous multiscale method for ordinary
  differential equations.
\newblock {\em Commun. Math. Sci.}, 1(3):423--436, 2003.

\bibitem{E-ENG:2003}
Weinan E and Bjorn Engquist.
\newblock Multiscale modeling and computation.
\newblock {\em Notices Amer. Math. Soc}, 50(50):1062--1070, 2003.

\bibitem{WeinanEngquistXiantao2007}
Weinan E, Bjorn Engquist, Xiantao Li, Weiqing Ren, and Eric Vanden-Eijnden.
\newblock Heterogeneous multiscale methods: A review.
\newblock {\em Communications in Computational Physics}, 2(3):367--450, June
  2007.

\bibitem{Engblom:2009}
S.~Engblom.
\newblock Parallel in time simulation of multiscale stochastic chemical
  kinetics.
\newblock {\em Multiscale Modeling and Simulation}, 8(1):46--68, 2009.

\bibitem{EngquistTsai:2005}
Bjorn Engquist and Yen-Hsi Tsai.
\newblock Heterogeneous multiscale methods for stiff ordinary differential
  equations.
\newblock {\em Mathematics of Computation}, 74(252):pp. 1707--1742, 2005.

\bibitem{FarhatChandesris:2003}
Charbel Farhat and Marion Chandesris.
\newblock Time-decomposed parallel time-integrators: theory and feasibility
  studies for fluid, structure, and fluid?structure applications.
\newblock {\em International Journal for Numerical Methods in Engineering},
  58(9):1397--1434, 2003.

\bibitem{Fischer04aparareal}
Paul~F. Fischer, Frédéric Hecht, and Yvon Maday.
\newblock A parareal in time semi-implicit approximation of the navier-stokes
  equations.
\newblock In {\em Proceedings of Fifteen International Conference on Domain
  Decomposition Methods}, pages 433--440. Springer Verlag, 2004.

\bibitem{Gander:2008}
Martin~J. Gander.
\newblock Analysis of the parareal algorithm applied to hyperbolic problems
  using characteristics.
\newblock {\em Bol. Soc. Esp. Mat. Apl. S$\vec{\rm e}$MA}, (42):21--35, 2008.

\bibitem{Gander}
MartinJ. Gander and Ernst Hairer.
\newblock Nonlinear convergence analysis for the parareal algorithm.
\newblock In Ulrich Langer, Marco Discacciati, DavidE. Keyes, OlofB. Widlund,
  and Walter Zulehner, editors, {\em Domain Decomposition Methods in Science
  and Engineering XVII}, volume~60 of {\em Lecture Notes in Computational
  Science and Engineering}, pages 45--56. Springer Berlin Heidelberg, 2008.

\bibitem{HairerLubichWanner:2010}
Ernst Hairer, Christian Lubich, and Gerhard Wanner.
\newblock {\em Geometric numerical integration}, volume~31 of {\em Springer
  Series in Computational Mathematics}.
\newblock Springer, Heidelberg, 2010.
\newblock Structure-preserving algorithms for ordinary differential equations,
  Reprint of the second (2006) edition.

\bibitem{Jones1999}
D.a. Jones, a.~Mahalov, and B.~Nicolaenko.
\newblock {A Numerical Study of an Operator Splitting Method for Rotating Flows
  with Large Ageostrophic Initial Data}.
\newblock {\em Theoretical and Computational Fluid Dynamics}, 13(2):143, 1999.

\bibitem{KassamTrefethen:2005}
Aly khan Kassam, Lloyd, and N.~Trefethen.
\newblock Fourth-order time stepping for stiff pdes.
\newblock {\em SIAM J. Sci. Comput}, 26:1214--1233, 2005.

\bibitem{Lawson:1967}
J.~Douglas Lawson.
\newblock Generalized runge-kutta processes for stable systems with large
  lipschitz constants.
\newblock {\em SIAM Journal on Numerical Analysis}, 4(3):pp. 372--380, 1967.

\bibitem{Legoll2012b}
F.~{Legoll}, T.~{Lelievre}, and G.~{Samaey}.
\newblock {A micro-macro parareal algorithm: application to singularly
  perturbed ordinary differential equations}.
\newblock {\em ArXiv e-prints}, April 2012.

\bibitem{Legoll2012}
Fr\'{e}d\'{e}ric Legoll, Xiaoying Dai, Claude {Le Bris}, and Yvon Maday.
\newblock {Symmetric parareal algorithms for Hamiltonian systems}.
\newblock {\em ESAIM: Mathematical Modelling and Numerical Analysis}, pages
  1--56, September 2012.

\bibitem{LI-MA-TU}
Jacques-Louis Lions, Yvon Maday, and Gabriel Turinici.
\newblock {R\'{e}solution d'\{EDP\} par un sch\'{e}ma en temps
  ``parar\'{e}el''}.
\newblock {\em C. R. Acad. Sci. Paris S\'{e}r. I Math.}, 332(7):661--668, 2001.

\bibitem{Maday:2007}
Yvon Maday.
\newblock Parareal in time algorithm for kinetic systems based on model
  reduction.
\newblock In {\em High-dimensional partial differential equations in science
  and engineering}, volume~41 of {\em CRM Proc. Lecture Notes}, pages 183--194.
  Amer. Math. Soc., Providence, RI, 2007.

\bibitem{MadayTurinici2003}
Yvon Maday and Gabriel Turinici.
\newblock Parallel in time algorithms for quantum control: Parareal time
  discretization scheme.
\newblock {\em International Journal of Quantum Chemistry}, 93(3):223--228,
  2003.

\bibitem{MAJDA:2003}
Andrew Majda.
\newblock {\em {Introduction to \{PDE\}s and waves for the atmosphere and
  ocean}}, volume~9 of {\em Courant Lecture Notes in Mathematics}.
\newblock New York University Courant Institute of Mathematical Sciences, New
  York, 2003.

\bibitem{Majda1998}
Andrew~J. Majda and Pedro Embid.
\newblock {Averaging over Fast Gravity Waves for Geophysical Flows with
  Unbalanced Initial Data}.
\newblock {\em Theoretical and Computational Fluid Dynamics}, 11(3-4):155--169,
  June 1998.

\bibitem{NadigaHechtMargolinSmolarkiewicz:1997}
B~Nadiga, Matthew Hecht, L~Margolin, and Piotr Smolarkiewicz.
\newblock On simulating flows with multiple time scales using a method of
  averages.
\newblock {\em Theoretical and Computational Fluid Dynamics}, 9(3-4):281--292,
  1997.

\bibitem{PetcuTemamWirosoetisno:2005}
M.~Petcu, R.~Temam, and D.~Wirosoetisno.
\newblock Renormalization group method applied to the primitive equations.
\newblock {\em J. Differential Equations}, 208(1):215--257, 2005.

\bibitem{Reich1995}
Sebastian Reich.
\newblock {Smoothed dynamics of highly oscillatory Hamiltonian systems},
  December 1995.

\bibitem{SandersVerhulst:1985}
J.~A. Sanders and F.~Verhulst.
\newblock {\em Averaging methods in nonlinear dynamical systems / J.A. Sanders,
  F. Verhulst}.
\newblock Springer-Verlag, New York:, 1985.

\bibitem{schochet1994}
S.~Schochet.
\newblock Fast singular limits of hyperbolic pde's.
\newblock {\em Journal of Differential Equations}, 114:476--512, 1994.

\bibitem{Smith2005}
Leslie~M. Smith and Youngsuk Lee.
\newblock {On near resonances and symmetry breaking in forced rotating flows at
  moderate Rossby number}.
\newblock {\em Journal of Fluid Mechanics}, 535(2005):111--142, July 2005.

\bibitem{Staff:2003}
G.~A. Staff.
\newblock The parareal algorithm. a survey of present work.
\newblock Technical report, Norwegian University of Science and Technology,
  Dept. of Math. Sciences, 2003.

\bibitem{Wingate2011}
Beth~a. Wingate, Pedro Embid, Miranda Holmes-Cerfon, and Mark~a. Taylor.
\newblock {Low Rossby limiting dynamics for stably stratified flow with finite
  Froude number}.
\newblock {\em Journal of Fluid Mechanics}, 676(2011):546--571, April 2011.

\end{thebibliography}

\end{document}